\documentclass[11pt,reqno]{amsart}

\usepackage{amsmath}								%math
\usepackage{amssymb}
\usepackage{amsthm}
\usepackage{amscd}
\usepackage{amsfonts}
\usepackage{stmaryrd}

\usepackage{euler}

\usepackage{extarrows}

\usepackage[colorlinks, linktocpage, citecolor = blue, linkcolor = blue]{hyperref}
\usepackage{color}
\usepackage{tikz}									
\usetikzlibrary{matrix}
\usetikzlibrary{patterns}
\usetikzlibrary{matrix}
\usetikzlibrary{positioning}
\usetikzlibrary{decorations.pathmorphing}

\usepackage{array}
\usepackage{float}
\usepackage{fullpage}
\usepackage{enumerate}

\newtheorem*{theorem*}{Theorem}
									
\newtheorem{theorem}{Theorem}[section]
\newtheorem{lemma}[theorem]{Lemma}
\newtheorem{proposition}[theorem]{Proposition}
\newtheorem{corollary}[theorem]{Corollary} 

\theoremstyle{definition}
\newtheorem{definition}[theorem]{Definition}
\newtheorem{example}[theorem]{Example}
\newtheorem{examples}[theorem]{Examples}

\newtheorem{remark}[theorem]{Remark}

\newcommand{\R}{\mathbb{R}}
\newcommand{\Rbar}{\overline{\mathbb{R}}}
\newcommand{\Z}{\mathbb{Z}}

\newcommand{\N}{\mathbb{N}}
\newcommand{\PP}{\mathbb{P}}
\newcommand{\NN}{\mathbb{N}}

\newcommand{\A}{\mathbb{A}}
\newcommand{\G}{\mathbb{G}}

\newcommand{\Sigmabar}{\overline{\Sigma}}
\newcommand{\sigmabar}{\overline{\sigma}}
\newcommand{\Mbar}{\overline{M}}
\newcommand{\calObar}{\overline{\mathcal{O}}}
\newcommand{\calA}{\mathcal{A}}

\newcommand{\calO}{\mathcal{O}}

\newcommand{\calMbar}{\overline{\mathcal{M}}}
\newcommand{\frakS}{\mathfrak{S}}

\DeclareMathOperator{\Spec}{Spec}
\DeclareMathOperator{\Hom}{Hom}

\DeclareMathOperator{\Aut}{Aut}

\DeclareMathOperator{\trop}{trop}
\DeclareMathOperator{\val}{val}

\DeclareMathOperator{\Trop}{Trop}

\DeclareMathOperator{\pr}{pr}
\DeclareMathOperator{\id}{id}

\DeclareMathOperator{\Span}{Span}

\title[Functorial tropicalization of logarithmic schemes]{Functorial tropicalization of logarithmic schemes: the case of constant coefficients}

\author{Martin Ulirsch}
\address{University of Michigan, Ann Arbor, MI 48109, USA}
\email{\href{mailto:ulirsch@umich.edu}{ulirsch@umich.edu}}
\date{\today}\thanks{M.U.'s research was supported in part by funds from BSF grant 201025, NSF grants DMS0901278 and DMS1162367, and by the SFB/TR 45 'Periods, Moduli Spaces and Arithmetic of Algebraic Varieties' of the DFG (German Research Foundation) as well as by the Hausdorff Center for Mathematics at the University of Bonn.} 

\begin{document}

\maketitle

\begin{abstract} The purpose of this article is to develop foundational techniques from logarithmic geometry in order to define a functorial tropicalization map for fine and saturated logarithmic schemes in the case of constant coefficients. Our approach crucially uses the theory of fans in the sense of K. Kato and generalizes Thuillier's retraction map onto the non-Archimedean skeleton in the toroidal case. For the convenience of the reader many examples as well as an introductory treatment of the theory of Kato fans are included.
\end{abstract}

\setcounter{tocdepth}{1}
\tableofcontents

%%%%%%%%%%%%%%%%%%%%%%%%%%%%%%%%%%%%%%%%%%%%%%%%%%%%%%

\section*{Introduction}
Tropical geometry associates to an algebraic variety $Y$ a "polyhedral shadow" known as its \emph{tropicalization}, whose polyhedral geometry surprisingly often corresponds to the algebraic geometry of $Y$. Classically (see e.g. \cite{MaclaganSturmfels_book}, \cite{Kajiwara_troptoric}, and \cite{Payne_anallimittrop}), in order to define the tropicalization of $Y$, one has to choose an embedding of $Y$ into a suitable toric variety $X$, and, in general, the geometry the tropicalization of $Y$ strongly depends on the chosen embedding $Y\hookrightarrow X$. 

Unfortunately, there are many varieties that do not admit a well-understood embedding into a toric variety, such as the Deligne-Knudsen-Mumford \emph{moduli spaces} $\calMbar_{g,n}$ of stable $n$-marked curves for $g>0$ (see \cite{DeligneMumford_moduliofcurves} and \cite{Knudsen_projectivityII}) or the \emph{toric degenerations} arising in the Gross-Siebert approach to mirror symmetry (see e.g. \cite{GrossSiebert_realaffinetocplx}, \cite{GrossSiebert_mirrorsymmetryvialogdegdataI}, and \cite{GrossSiebert_mirrorsymmetryvialogdegdataII}). A common feature of these varieties is that they are locally (in the \'etale topology) isomorphic to toric varieties, i.e. they canonically carry the structure of a \emph{toroidal embedding} in the sense of \cite{KKMSD_toroidal}. 

In this article we use techniques from Berkovich analytic spaces (see e.g. \cite{Berkovich_book} and \cite{Berkovich_etalecoho}) and logarithmic geometry (see \cite{Kato_logstr} and \cite{Kato_toricsing}) to construct a natural tropicalization map associated to any fine and saturated logarithmic scheme $X$ that is locally of finite type over a trivially valued base field. We, in particular, show that this tropicalization map is functorial with respect to logarithmic morphisms and that it recovers Thuillier's \cite{Thuillier_toroidal} strong deformation retraction onto the non-Archimedean skeleton of $X$ in the logarithmicallly smooth case (see Theorems \ref{thm_tropfunc} and  \ref{thm_skel=trop} below). Using these techniques we derive a criterion for a sch\"on variety (in the sense of \cite{Tevelev_tropcomp}) to admit a faithful tropicalization (see Corollary \ref{cor_schoen&faithful} below). 

The main applications of the techniques developed in this article lie in the tropical geometry of moduli spaces.  In \cite{AbramovichCaporasoPayne_tropicalmoduli}, as an archetypical result, the authors show that the moduli space  of stable tropical curves is isomorphic to the skeleton of the moduli space of algebraic curves with the respect to the toroidal structure coming from the Deligne-Knudsen-Mumford compactification (see \cite{DeligneMumford_moduliofcurves} and \cite{Knudsen_projectivityII}). Similar results have appeared for other moduli spaces, such as the moduli space of admissible covers (see \cite{CavalieriMarkwigRanganathan_admissiblecovers}), the moduli space of weighted stable curves (see \cite{CavalieriMarkwigHampeRanganathan_weightedstablecurves} for the case genus $g=0$ and \cite{Ulirsch_weightedstablecurves} for $g\geq 0$), and the moduli space of rational logarithmic stable maps into a toric variety (see \cite{Ranganathan_ratcurtorvarnonArch}). 

Let us now give a short outline of this article: In Section \ref{section_mainresults} we give an overview of the historical background, provide precise statements of our results, and discuss further developments that have appeared in parallel or after this article has first been made available. Section \ref{section_monoidscones&monoidalspaces} introduces the basic notions of monoids and cones and the theory of locally monoidal spaces. In Section \ref{section_Katofans&conecomplexes} we introduce the technical heart of our construction, the notion of a Kato fan, as originally proposed in \cite{Kato_toricsing}, and connect it with the theory of (extended) rational polyhedral cone complexes. Section \ref{section_logstr} gives a quick introduction to the theory of logarithmic structures in the sense of Kato-Fontaine-Illusie (see \cite{Kato_logstr}) and shows how Kato fans naturally arise in this theory. Section \ref{section_analytification} introduces the two different non-Archimedean analytification functors used in this text and explains their differences using the torus-invariant open subsets of the projective line as explicit examples. In Section \ref{section_tropicalization} we construct the tropicalization map and prove our main results. Section \ref{section_toricvar}, finally, is concerned with a comparison of our construction with the well-known embedded tropicalization for (subvarieties of) toric varieties in the sense of Kajiwara and Payne (see \cite{Kajiwara_troptoric} and \cite{Payne_anallimittrop}) and a proof of Corollary \ref{cor_schoen&faithful}.

%%%%%%%%%%%%%%%%%%%%%%%%%%%%%%%%%%%%%%%%%%%%%%%%%%%%%%

\subsection*{Acknowledgements}

The author would like to express his gratitude to Dan Abramovich for his constant support and encouragement. Thanks are also due to Walter Gubler and Sam Payne, as well as to Amaury Thuillier, who originally suggested to the author the use of Kato fans in order to define tropicalization maps, and to Jonathan Wise, whose ideas related to Artin fans heavily influenced this work. The author also profited from discussions with Lorenzo Fantini, Jeffrey and Noah Giansiracusa, Angela Gibney, Andreas Gross, Alana Huszar, Oliver Lorscheid, Diane Maclagan, Steffen Marcus, Samouil Molcho, Johannes Nicaise, Joseph Rabinoff, Dhruv Ranganathan, Mattia Talpo, and Michael Temkin. Parts of this research have been carried out while enjoying the hospitality of Hebrew University, Jerusalem, and the University of Regensburg. Particular thanks are due to the anonymous referee(s) for several suggestions that significantly improved the readability of this article. 

%%%%%%%%%%%%%%%%%%%%%%%%%%%%%%%%%%%%%%%%%%%%%%%%%%%%%%

\section{Overview and statement of the main results}\label{section_mainresults}

\subsection{Historical background: Tropicalization of tori, toric varieties, and toroidal embeddings} 

Let $k$ be a field that is endowed with a (possibly trivial) non-Archimedean absolute value $\vert.\vert$, let $N$ be a finitely generated free abelian group of dimension $n$, and write $M$ for its dual $\Hom(N,\Z)$ as well as $\langle.,.\rangle$ for the duality pairing between $N$ and $M$. 

\subsubsection{} \emph{Tropicalization} is a process that associates to a closed subset $Y$ of the split algebraic torus $T=\Spec k[M]$ a subset $\Trop(Y)$ in $N_\R=N\otimes\R$, the \emph{tropicalization} of $Y$, which can be endowed with the structure of a rational polyhedral complex. Following \cite{EinsiedlerKapranovLind_amoebas} as well as \cite{Gubler_tropvar} and \cite{Gubler_guide}, one of the many ways of defining $\Trop(Y)$ is by taking it to be the projection of the non-Archimedean analytic space $Y^{an}$ (see Section \ref{section_analytification} below) associated to $Y$ into $N_\R$ via a natural continuous \emph{tropicalization map} 
\begin{equation*}
\trop\mathrel{\mathop:}T^{an}\longrightarrow N_\R \ .
\end{equation*} 
The image $\trop(x)$ of a point $x\in T^{an}$, given by a seminorm $\vert.\vert_x\mathrel{\mathop:}k[M]\rightarrow \R$ extending  $\vert.\vert$ on $k$, is uniquely determined by the condition
\begin{equation*}
\langle\trop(x),m\rangle=-\log\big\vert\chi^m\big\vert_x
\end{equation*}
for all $m\in M$, where $\chi^m$ denotes the character in $k[M]$ corresponding to $m\in M$. 

\subsubsection{}\label{section_introtroptoric} Kajiwara \cite[Section 1]{Kajiwara_troptoric} and, independently, Payne \cite[Section 3]{Payne_anallimittrop} define a natural continuous extension of the above tropicalization map to a $T$-toric variety $X=X(\Delta)$ defined by a rational polyhedral fan $\Delta$ in $N_\R$. For standard notation and general background on toric varieties we refer the reader to \cite{Fulton_toricvarieties}. 

The codomain of the tropicalization map is a partial compactification $N_\R(\Delta)$ of $N_\R$ determined by $\Delta$, a detailed construction of which can be found in \cite[Section 3]{Rabinoff_newtonpolygon}. For a torus-invariant open affine subset $U_\sigma=\Spec k[S_\sigma]$ of $X$ defined by the semigroup $S_\sigma=\sigma^\vee\cap M$ for a cone $\sigma\in\Delta$, the compactification $N(\sigma)$ is given by $\Hom(S_\sigma,\Rbar)$, where $\Rbar=\R\cup\{\infty\}$ is endowed with the natural additive monoid structure, and the tropicalization map 
\begin{equation*}
\trop_\Delta\mathrel{\mathop:}U_\sigma^{an}\longrightarrow N_\R(\sigma) 
\end{equation*} 
sends $x\in U_\sigma^{an}$ to the element $\trop_\Delta(x)\in\Hom(S_\sigma,\Rbar)$ that is determined by 
\begin{equation*}
\trop_\Delta(x)(s)=-\log\vert\chi^s\vert_x 
\end{equation*} 
for all $s\in S_\sigma$. From an alternative point of view, one may also think of $\trop_\Delta$ as a \emph{non-Archimedean analytic moment map} (see \cite[Remark 3.3]{Payne_anallimittrop}, \cite{Kajiwara_troptoric}, \cite[Section 4.1]{GilPhilipponSombra_arithmetictoric}, and \cite[Sections 1.1 and 1.2]{Ulirsch_tropisquot}). 

\subsubsection{} Suppose that $\vert.\vert$ is the trivial absolute value, i.e. $\vert a\vert=1$ for all $a\in k^\ast$. In this case Thuillier \cite[Section 2]{Thuillier_toroidal} constructs a closely related strong deformation retraction 
\begin{equation*}
\mathbf{p}\mathrel{\mathop:}X^\beth\longrightarrow X^\beth 
\end{equation*}
from $X^\beth$ onto the \emph{non-Archimedean skeleton} $\mathfrak{S}(X)$ of $X$ using the natural action of the analytic group $T^\beth$ on $X^\beth$. By \cite[Proposition 2.9]{Thuillier_toroidal} there is a natural embedding $i_\Delta\mathrel{\mathop :}\mathfrak{S}(X)\hookrightarrow N_\R(\Delta)$ such that the diagram
\begin{equation*}\begin{CD}
X^\beth @>\mathbf{p}>>\mathfrak{S}(X)\\
@V\subseteq VV @VVi_\Delta V\\
X^{an}@>\trop_\Delta>>N_\R(\Delta)
\end{CD}\end{equation*}
is commutative and the image of $\mathfrak{S}(X)$ in $N_\R(\Delta)$ is the closure $\overline{\Delta}$ of $\Delta$ in $N_\R(\Delta)$. Note that on a torus-invariant open affine subset $U_\sigma=\Spec k[S_\sigma]$ of $X$, for a cone $\sigma\in\Delta$, the image of $i_\Delta$ is given by
 \begin{equation*}
 \overline{\sigma}=\Hom(S_\sigma,\Rbar_{\geq 0})\subseteq\Hom(S_\sigma,\Rbar) \ ,
\end{equation*}
which is also known as the \emph{canonical compactifcation} of the cone $\sigma=\Hom(S,\R_{\geq 0})$.

%%%%%%%%%%%%%%%%%%%%%%%%%%%%%%%%%%%%%%%%%%%%%%%%%%%%%%

\subsubsection{} Suppose now that $X_0\hookrightarrow X$ is a toroidal embedding, i.e. an open and dense embedding that is \'etale locally isomorphic to the open embedding of a big algebraic torus into a toric variety. Using formal torus actions Thuillier \cite[Section 3]{Thuillier_toroidal} is able to lift his construction for toric varieties and obtains a strong deformation retraction
\begin{equation*}
\mathbf{p}\mathrel{\mathop:}X^\beth\longrightarrow X^\beth
\end{equation*}
onto the \emph{non-Archimedean skeleton} $\mathfrak{S}(X)$ of $X$. We refer to \cite[Chapter 2]{KKMSD_toroidal} and the beginning of \cite[Section 3]{Thuillier_toroidal} for the basic theory of toroidal embeddings. In \cite{KKMSD_toroidal} the authors work with formal instead of \'etale neighborhoods, but by \cite[Section 2]{Denef_toroidal} both approaches are equivalent over algebraically closed fields.

In \cite{AbramovichCaporasoPayne_tropicalmoduli} Abramovich, Caporaso, and Payne explain how $\mathfrak{S}(X)$ can be endowed with the structure of a \emph{generalized extended cone complex} $\Sigmabar_X$. By \cite[Proposition 3.15]{Thuillier_toroidal}, if $X_0\hookrightarrow X$ has \emph{no self-intersection} in the terminology of \cite{KKMSD_toroidal}, then $\Sigmabar_X$ is the canonical compactification of the rational polyhedral cone complex $\Sigma_X$ associated to the toroidal embedding as constructed in the end of \cite[Section 2.1]{KKMSD_toroidal}. Note that other authors also use the adjectives \emph{simple} or \emph{strict} to denote toroidal embeddings without self-intersection. 

\subsection{Tropicalization of logarithmic schemes}  Let $X$ be a fine and saturated logarithmic scheme locally of finite type over $k$. In this article we associate to $X$ a \emph{generalized cone complex} $\Sigma_X$, expanding on both \cite{AbramovichCaporasoPayne_tropicalmoduli} and \cite{KKMSD_toroidal}, and construct a natural continuous tropicalization map $\trop_X\mathrel{\mathop:}X^\beth\rightarrow \Sigmabar_X$ from $X^\beth$ into its canonical extension $\Sigmabar_X$. 
For our definition to be reasonable we require $\trop_X$ to fulfill the following two properties:
\begin{enumerate}[(i)]
\item The tropicalization map is functorial with respect to logarithmic morphisms.
\item In the logarithmically smooth case $\trop_X$ recovers Thuillier's retraction map.
\end{enumerate}
To be precise, the following two theorems have to hold:

\begin{theorem}\label{thm_tropfunc}
A morphism $f\mathrel{\mathop:}X\rightarrow X'$ of fine and saturated logarithmic schemes locally of finite type over $k$ induces a morphism $\Sigma(f)\mathrel{\mathop:}\Sigma_X\rightarrow\Sigma_{X'}$ of generalized cone complexes that makes the induced diagram
\begin{equation*}\begin{CD}
X^\beth @>\trop_X>>\overline{\Sigma}_X\\
@Vf^\beth VV @VV\Sigmabar(f) V\\
(X')^\beth @>\trop_{X'}>> \overline{\Sigma}_{X'} 
\end{CD}\end{equation*}
commute. The association $f\mapsto \Sigma(f)$ is functorial in $f$.
\end{theorem}

Suppose now that $X$ is logarithmically smooth over $k$. Then $X$ has the structure of a toroidal embedding (see Section \ref{section_logsmooth} below).

\begin{theorem}\label{thm_skel=trop} If $X$ is logarithmically smooth, then $\trop_X$ has a section onto the skeleton of $X$, i.e. there is a homeomorphism $J_X\mathrel{\mathop:}\Sigmabar_X\xrightarrow{\sim} \mathfrak{S}(X)$ such that the diagram
\begin{center}\begin{tikzpicture}
  \matrix (m) [matrix of math nodes,row sep=2em,column sep=3em,minimum width=2em,ampersand replacement=\&]
  {  
  \& X^\beth \& \\ 
  \mathfrak{S}(X)  \& \& \Sigmabar_X  \\ 
  };
  \path[-stealth]
    (m-1-2) edge node [above left] {$\mathbf{p}_X$} (m-2-1)
    		edge node [above right] {$\trop_X$} (m-2-3)
    (m-2-3) edge node [below] {$J_X$} node [above] {$\sim$} (m-2-1);		
\end{tikzpicture}\end{center}
is commutative.
\end{theorem}

\subsection{Tropicalizing subvarieties} Let $X$ be a logarithmic scheme locally of finite type over $k$ and let $Y\subseteq X$ be a closed subvariety. One may define the \emph{tropicalization} associated to $Y$ with respect to $X$ by setting
\begin{equation*}
\Trop_X(Y)=\trop_X(Y^\beth) \ .
\end{equation*}

In \cite[Theorem 1.1]{Ulirsch_tropcomplogreg} we have shown that, if $X$ is a Zariski logarithmic scheme that is logarithmically smooth over $k$ and if $Y$ non-trivially intersects the open locus $X_0$ of $X$ where the logarithmic structure is trivial, then $\Trop_X(Y)\cap\Sigma_X$ also carries the structure of a cone complex. Moreover, the further results of \cite{Ulirsch_tropcomplogreg} show that one can use the polyhedral geometry of $\Trop_X(Y)$ to study properties of $Y$, thought of as a partial embedded compactification of $Y_0=Y\cap X$, expanding on Tevelev's theory of \emph{tropical compactifications} (see \cite{Tevelev_tropcomp}). 

Let $Y$ be a closed subvariety of a $T$-toric variety $X$ and assume $Y\cap T\neq\emptyset$. Then there are two different ways to tropicalize $Y$: the first is to consider the classical Kajiwara-Payne tropicalization $\Trop_\Delta(Y)$ (see Section \ref{section_introtroptoric} above); the other is to endow $Y$ with the pullback of the logarithmic structure from $X$ and to consider its image in $\Sigmabar_X$. In Section \ref{section_toricvar} we give a detailed comparison of these two cases. These considerations together with the above two theorems lead to the following incarnation of the principle of \emph{faithful tropicalization} (see e.g. \cite{BakerPayneRabinoff_nonArchtrop}, \cite{GublerRabinoffWerner_skeletons&trop}, and \cite{CuetoHaebichWerner_tropGrass}).

\begin{corollary}\label{cor_schoen&faithful}
Suppose that $Y$ is a proper and sch\"on subvariety of a $T$-toric variety $X$ such that $Y\cap T\neq\emptyset$. Then the restriction 
\begin{equation*}
\trop_\Delta\vert_{Y^{an}}\mathrel{\mathop:} Y^{an}\longrightarrow \Trop_\Delta(Y)
\end{equation*}
of the Kajiwara-Payne tropicalization map has a unique continuous section 
\begin{equation*}
J_Y\mathrel{\mathop:}\Trop_\Delta(Y)\longrightarrow Y^{an}
\end{equation*} 
such that  $J_Y\circ\trop_\Delta\mathrel{\mathop:}Y^{an}\rightarrow \frakS(Y)$ is the deformation retraction onto the toroidal skeleton if and only if the intersection of $Y$ with every $T$-orbit in $X$ is non-empty and irreducible (i.e. has multiplicity one).
\end{corollary}

%%%%%%%%%%%%%%%%%%%%%%%%%%%%%%%%%%%%%%%%%%%%%%%%%%%%%%%%%

\subsection{The main idea of our construction} Our approach to the construction of $\trop_X$ is a global version of the \emph{local tropicalization map} defined by Popescu-Pampu and Stepanov in \cite[Section 6]{PopescuPampuStepanov_localtrop}. Fix a morphism $\alpha\mathrel{\mathop:}P\rightarrow A$ from a monoid $P$ into the multiplicative monoid of an algebra $A$ of finite type over $k$ (or more generally into the quotient $A/A^\ast$) and set $X=\Spec A$. The set $\sigmabar_P=\Hom(P,\Rbar_{\geq 0})$ is the canonical compactification of the rational polyhedral cone $\sigma_P=\Hom(P,\R_{\geq 0})$ and there is a natural continuous tropicalization map 
\begin{equation*}
\trop_\alpha\mathrel{\mathop:}X^\beth\longrightarrow\sigmabar_P
\end{equation*}
that is defined by sending $x\in X^\beth$ to the homomorphism $\trop_\alpha(x)\in\sigmabar_P=\Hom(P,\Rbar_{\geq 0})$ given by
\begin{equation*}
p\longmapsto -\log\big\vert\alpha(p)\big\vert_x
\end{equation*}
for $p\in P$. 

In the global case we proceed in two steps. Associated to a Zariski logarithmic scheme $X$ \emph{without monodromy}  (see Section \ref{section_monodromy} below) there is an essentially unique \emph{characteristic morphism} $\phi_X\mathrel{\mathop:}(X,\overline{\mathcal{O}}_X)\rightarrow F_X$ into a \emph{Kato fan} $F_X$, i.e. a locally monoidal space that is covered by affine patches $\Spec P$ for fine and saturated monoids $P$ (see \cite[Section 9]{Kato_toricsing} and Section \ref{section_Katofans} below). The set of $\Rbar_{\geq 0}$-valued points $F_X(\Rbar_{\geq 0})$ carries the structure of an extended cone complex $\Sigmabar_X$ and, as a heuristic, should be thought of as a non-Archimedean analytic space associated to $F_X$. The natural continuous tropicalization map 
\begin{equation*}
\trop_X\mathrel{\mathop:}X^\beth\longrightarrow\Sigmabar_X
\end{equation*}
is then formally defined as the "analytification" of the characteristic morphism $\phi_X$. %To be precise $\trop_X$ is given by sending a point $x\in X^\beth$, represented by a morphism $\underline{x}\mathrel{\mathop:}\Spec R\rightarrow (X,\overline{\mathcal{O}}_X)$, to the point $\trop_X(x)\in\Sigmabar_X=F(\Rbar_{\geq 0})$ that is given by the composition
%\begin{equation*}\begin{CD}
%\Spec\Rbar_{\geq 0}@>\val^\#>>\Spec R @>\underline{x}>> (X,\overline{\mathcal{O}}_X) @>\phi_X>> F_X \ ,
%\end{CD}\end{equation*}
%where $\val^\#$ denotes the morphism induced by the valuation $\val\mathrel{\mathop:}R\rightarrow\Rbar_{\geq 0}$ on $R$.
In the special case that both $X=\Spec A$ and $F_X=\Spec P$ are affine, we have $\Sigmabar_X=\sigmabar_P$  and the tropicalization map $\trop_X$ is nothing but the local tropicalization map of Popescu-Pampu and Stepanov \cite{PopescuPampuStepanov_localtrop}.

In general, not every logarithmic scheme $X$ admits a characteristic morphism, a phenomenon that is due to the presence of \emph{monodromy} in the logarithmic structure (see Section \ref{section_monodromy}). In these cases we may still define a generalized cone complex $\Sigma_X$ as well as a tropicalization map $\trop_X\mathrel{\mathop:}X^\beth\rightarrow \Sigmabar_X$ by taking colimits over all strict \'etale covers by logarithmic schemes without monodromy. 

%%%%%%%%%%%%%%%%%%%%%%%%%%%%%%%%%%%%%%%%%%%%%%%%%%%%%%

\subsection{Complements, applications, and further developments}

\subsubsection{} Our approach to the tropicalization of logarithmic schemes is very much inspired by ideas of M. Gross and Siebert that have been outlined in \cite[Appendix B]{GrossSiebert_logGromovWitten}; their construction is motivated by applications to logarithmic Gromov-Witten theory in analogy with the tropical part of an \emph{exploded manifold} in the sense of \cite{Parker_explosion}. Let $X$ be a fine and saturated Zariski logarithmic scheme locally of finite type over $k$. According to Gross and Siebert the tropicalization of $X$ is given as the space
\begin{equation*}
\Trop^{GS}(X)=\Big(\bigsqcup_{x\in X}\Hom(\overline{M}_{X,x},\R_{\geq 0})\Big)\Big/\sim \ ,
\end{equation*}
where the equivalence relation $\sim$ is induced by the duals of the generization maps $\overline{M}_{X,x'}\rightarrow\overline{M}_{X,x}$, whenever $x'$ is a specialization of $x$ in $X$. In fact, from our construction one can deduce that there is a natural homeomorphism
\begin{equation*}
\Trop^{GS}(X)\simeq\Sigma_X
\end{equation*} 
and $\Trop^{GS}(X)$ therefore naturally carries the structure of a generalized cone complex.

%%%%%%%%%%%%%%%%%%%%%%%%%%%%%%%%%%%%%%%%%%%%%%%%%%%%%%

\subsubsection{} An alternative approach to the tropicalization of subvarieties of toric varieties, that is similar in spirit to our construction, can be found in \cite{GiansiracusaGiansiracusa_tropicalschemes}. Let $N$ be a free finitely generated abelian group and consider a rational polyhedral fan $\Delta$ in $N_\mathbb{R}$. Then $\Delta$ defines a toric variety $X_\Delta$ over the field $\mathbb{F}_1$ with one element and the associated sharp monoidal space $(X,\overline{\mathcal{O}}_X)=(X,\mathcal{O}_X/\mathcal{O}_X^\ast)$ turns out to be a toric Kato fan. Instead of "analytifying" $F_X$ by considering the $\Rbar_{\geq 0}$-valued points of $X_\Delta$, the authors of \cite{GiansiracusaGiansiracusa_tropicalschemes} work in the category of semiring scheme and consider the base change $X_\Delta\times_{\mathbb{F}_1}\mathbb{T}$, where $\mathbb{T}$ denotes the semi-ring of tropical numbers. 

In \cite{Lorscheid_tropicalschemes}, based on his theory of blueprints (see \cite{Lorscheid_blueprintsI}), Lorscheid proposes a much larger framework for tropicalization that generalizes both Giansiracusa and Giansiracusa's theory of tropical schemes and the tropical geometry of logarithmic schemes without monodromy. Let $\mathbb{T}_{\geq 0}$ be the semiring of non-negative tropical numbers. Lorscheid, in particular, shows that both $X^\beth$ and $\Sigmabar_X$ (and more generally $Y^\beth$ and the tropicalization $\Trop_X(Y)$ for a closed subset $Y$ of $X$) naturally arise as the set of $\mathbb{T}_{\geq 0}$-valued points of certain suitable chosen so-called \emph{ordered blue schemes} (see \cite[Theorem H]{Lorscheid_tropicalschemes}). In this language the tropicalization map $\trop_X$, as defined in this article, arises as the induced map on $\mathbb{T}_{\geq 0}$-valued points of an underlying morphism of ordered blue schemes, which one can think of as an enrichment of the characteristic morphism $\phi_X$ used in our construction. 

%%%%%%%%%%%%%%%%%%%%%%%%%%%%%%%%%%%%%%%%%%%%%%%%%%%%%%

\subsubsection{} Let $X$ be a Zariski logarithmic scheme that is logarithmically smooth over the base field. In \cite{Gross_toroidalintersectiontheory} A. Gross has developed a version of tropical intersection theory on (extended) cone complexes that admit a weak embedding into a vector space generated by a lattice, expanding on the theory developed in this article. In particular, using his approach one can enrich the process of tropicalization to an operation on algebraic cycles on $X$ that non-trivially intersect $X_0$. 

In "good" cases, his approach also allows us to identify certain natural tropical intersection products with their algebraic counterparts. Applying these identities to the moduli space of rational logarithmic stable maps into a toric variety, we can deduce that certain rational tropical and algebraic (descendant) Gromov-Witten invariants agree (see e.g. \cite[Section 5]{Gross_toroidalintersectiontheory} and \cite[Theorem C]{Ranganathan_ratcurtorvarnonArch}), which provides us with a moduli-theoretic explanation of the classical Nishinou-Siebert correspondence theorem (see \cite{NishinouSiebert_correspondence}).

%%%%%%%%%%%%%%%%%%%%%%%%%%%%%%%%%%%%%%%%%%%%%%%%%%%%%%

\subsubsection{} In \cite{AbramovichChenMarcusWise_boundedness} and \cite{AbramovichWise_invlogGromovWitten} (also see \cite{AbramovichChenMarcusUlirschWise_logsurvey} and \cite{Ulirsch_Artinfans}), the authors develop the theory of \emph{Artin fans}, an incarnation of the theory of Kato fans in the category of logarithmic algebraic stacks that is more suitable to deal with logarithmic structures that have monodromy. In particuar, for every logarithmic scheme there is an Artin fan $\calA_X$ and an essentially unique strict morphism $X\rightarrow\calA_X$ that is a lift of the characteristic morphism to this category. For example, if $X$ is a $T$-toric variety, the Artin fan $\calA_X$ is the toric quotient stack $\big[X\big/T\big]$. 

In \cite[Theorem 1.1]{Ulirsch_Artinfans} we show that defining $\trop_X$ as the "analytification" of the characteristic morphism is much more than a mere heuristic: There is a natural homeomorphism $\mu_X\mathrel{\mathop:}\big\vert\mathcal{A}_X^\beth\big\vert\rightarrow \Sigmabar_X$ from the topological space $\big\vert\mathcal{A}_X^\beth\big\vert$ underlying the non-Archimedean analytic stack $\mathcal{A}_X^\beth$ with $\Sigmabar_X$ making the diagram
\begin{center}\begin{tikzpicture}
  \matrix (m) [matrix of math nodes,row sep=2em,column sep=3em,minimum width=2em,ampersand replacement=\&]
  {  
  \& X^\beth \& \\ 
  \big\vert\calA_X^\beth\big\vert  \& \& \Sigmabar_X  \\ 
  };
  \path[-stealth]
    (m-1-2) edge node [above left] {$\phi_X^\beth$} (m-2-1)
    		edge node [above right] {$\trop_X$} (m-2-3)
    (m-2-1) edge node [below] {$\mu_X$} node [above] {$\sim$} (m-2-3);		
\end{tikzpicture}\end{center}
commute. If $X$ is a $T$-toric variety, this statement generalizes to the fact that the Kajiwara-Payne tropicalization map $\trop_\Delta\mathrel{\mathop:}X^{an}\rightarrow N_\R(\Delta)$ is a stack quotient $\big[X^{an}\big/T^\circ\big]$, where $T^\circ$ denotes the \emph{affinoid torus} of $T$, a non-Archimedean version of $S^1\otimes N$. This procedure, in particular, gives every rational polyhedral cone complex $\Sigma$ as well as its canonical compactification $\Sigmabar$ canonically the structure of a non-Archimedean analytic stack. 

%%%%%%%%%%%%%%%%%%%%%%%%%%%%%%%%%%%%%%%%%%%%%%%%%%%%%%

%%%%%%%%%%%%%%%%%%%%%%%%%%%%%%%%%%%%%%%%%%%%%%%%%%%%%%

\section{Monoids, cones, and monoidal spaces}\label{section_monoidscones&monoidalspaces}

\subsection{Monoids} A \emph{monoid} $P$ is a commutative semigroup with an identity element. All monoids will be written additively, unless noted otherwise. The non-negative real numbers together with addition form a monoid that is denoted by $\R_{\geq 0}$. Its monoid structure naturally extends to $\Rbar_{\geq 0}=\R_{\geq 0}\cup\{\infty\}$ by setting $a+\infty = \infty$ for all $a\in \Rbar_{\geq 0}$.

An \emph{ideal} $I$ in a monoid $P$ is a subset $I\subseteq P$ such that $p+I\subseteq I$ for all $p\in P$. Every monoid $P$ contains a unique maximal ideal $\mathfrak{m}_P=P-P^\ast$. An ideal $\mathfrak{p}$ in $P$ is called \emph{prime} if its complement $P-\mathfrak{p}$ in $P$ is a submonoid, or equivalently, if $p_1+p_2\in \mathfrak{p}$ already implies $p_1\in \mathfrak{p}$ or $p_2\in \mathfrak{p}$ for all $p_i\in P$. The complement of a prime ideal in $P$ is referred to as a \emph{face} of $P$.

The \emph{localization} of a monoid $P$ with respect to a submonoid $S$ is given by \begin{equation*}
S^{-1}P=\{p-s\vert p\in P, \ s\in S\} \ ,
\end{equation*}
where $p-s$ denotes an equivalence class of pairs $(p,s)\in P\times S$ under the equivalence relation
\begin{equation*}
(p,s)\sim (p',s') \Leftrightarrow \exists\  t\in S \textrm{ such that } p+s'+t=p'+s+t \ .
\end{equation*}
If $S$ is the set $\mathbb{N}\cdot f$ for an element $f\in P$ we write $P_f$ for the localization $S^{-1}P$ and if $S$ is the complement of a prime ideal $\mathfrak{p}$ in $P$ we denote $S^{-1}P$ by $P_\mathfrak{p}$. 

A monoid $P$ is called \emph{fine}, if it is finitely generated and the canonical homomorphism into the group $P^{gp}=P^{-1}P=\{p-q\vert\  p,q\in P\}$ is injective. It is said to be \emph{saturated} if, whenever $p\in P^{gp}$, the property $n\cdot p\in P$ for some $n\in\mathbb{N}_{> 0}$ already implies $p\in P$. An element $p\in P$ is called a \emph{torsion element}, if $n\cdot p =0$ for some $n\in\mathbb{N}_{> 0}$; it is called a \emph{unit}, if there is $q\in P$ such that $p+q=0$. Denote the subgroup of torsion elements in $P$ by $P^{tors}$ and the subgroup of units by $P^\ast$. A fine and saturated monoid $P$ is said to be \emph{toric}, if $P^{tors}=0$; any monoid $P$ is said to be \emph{sharp}, if $P^\ast=0$. We denote the category of fine and saturated by $\mathbf{fs-Mon}$ and the full subcategory of toric monoids by $\mathbf{tor-Mon}$.

\begin{lemma}\label{lemma_fsmondecomp}
Let $P$ be a fine and saturated monoid. 
\begin{enumerate}[(i)]
\item There is a toric submonoid $\widetilde{P}$ of $P$ such that $P=\tilde{P}\oplus P^{tors}$.
\item There exists a sharp submonoid $\overline{P}$ of $P$ such that $P=\overline{P}\oplus P^\ast$. 
\end{enumerate}
\end{lemma}

For the convenience of the reader we provide proofs of these two well-known statements.

\begin{proof}[Proof of Lemma \ref{lemma_fsmondecomp}]
The abelian group $Q=P^{gp}$ is finitely generated. So we can find a finitely generated free abelian subgroup $\widetilde{Q}$ of $Q$ such that $Q=\widetilde{Q}\oplus Q^{tors}$. Note that hereby $Q^{tors}=P^{tors}$, since $n\cdot q=0\in P$ for $q\in Q$ and some $n\in\N_{>0}$ already implies $q\in P$ using that $P$ is saturated.
Set $\widetilde{P}=P\cap\tilde{Q}$. Every $p\in P$ can be uniquely written as $\tilde{p}+t$ with $\tilde{p}\in \widetilde{Q}$ and $t\in P^{tors}$ and we have $\tilde{p}=p-t\in P$. Thus $P=\tilde{P}\oplus P^{tors}$. This proves part (i).

In view of (i), we may assume $P^\ast=0$ for the proof of part (ii). Given $q\in Q$ such that $n\cdot q\in P^\ast$ for some $n\in\N_{>0}$, we already have $q\in P$, since $P$ is saturated. Therefore $P^\ast$ is a saturated abelian subgroup in $Q$, i.e. $Q/P^\ast$ is free, and we can find a subgroup $\overline{Q}$ of $Q$ such that $Q=\overline{Q}\oplus P^\ast$. So every element $p\in P$ can be uniquely written as $\overline{p}+u$ with $\overline{p}\in \overline{Q}$ and $u\in P^\ast$. Set $\overline{P}=P\cap\overline{Q}$. Since $\overline{p}=p-u\in P$, this implies $P=\overline{P}\oplus P^\ast$.
\end{proof}

In a slight abuse of notation, we write $\widetilde{P}$ for the toric monoid $P/P^{tors}$ and $\overline{P}$ for the sharp monoid $P/P^\ast$.

\subsection{Rational polyhedral cones} A \emph{strictly convex rational polyhedral cone} (or short: a \emph{rational polyhedral cone}) is a pair $(\sigma,N)$ consisting of a finitely generated free abelian group $N$ and a strictly convex rational polyhedral cone $\sigma\subseteq N_\R=N\otimes\R$, i.e. a finite intersection of half spaces 
\begin{equation*}
H_i=\big\{u\in N_\R\big\vert\langle u,v_i\rangle\geq 0\big\}  \ ,
\end{equation*}
where $v_i\in M$ such that $\sigma$ does not contain any non-trivial linear subspaces. We refer to   \cite[Section 1.2]{Fulton_toricvarieties} and \cite[Appendix A]{Gubler_guide} for the essential background on these notions. Note hereby that Gubler \cite{Gubler_guide} calls rational polyhedral cones \emph{pointed integral polyhedral cones}. We denote the \emph{relative interior} of a rational polyhedral cone $\sigma$, i.e. the interior of $\sigma$ in its span in $N_\R$, by $\mathring{\sigma}$. A morphism $f\mathrel{\mathop:}(\sigma,N)\rightarrow(\sigma', N')$ of rational polyhedral cones is given by an element $f\in\Hom(N,N')$ such that $f(\sigma)\subseteq\sigma'$. 

Consider now the functor $\sigma$ on $\mathbf{fs-Mon}^{op}$ that associates to a fine and saturated monoid $P$ the rational polyhedral cone $(\sigma_P,N_P)$ given by 
\begin{equation*}
N_P=\Hom(P^{gp},\Z)
\end{equation*}
and
\begin{equation*}
\sigma_P=\Hom(P,\R_{\geq 0})
=\big\{u\in\Hom(P^{gp},\R)\big\vert u(p)\geq 0 \text{ for all } p\in P\big\}\subseteq (N_P)_\R \ .
\end{equation*}
It is immediate that a morphism $f\mathrel{\mathop:}Q\rightarrow P$ of fine and saturated monoids $P$ and $Q$ induces a morphism $\sigma(f)\mathrel{\mathop:}(\sigma_P,N_P)\rightarrow(\sigma_Q,N_Q)$ and that the association $f\mapsto \sigma(f)$ is functorial in $f$.

\begin{proposition}\label{prop_rpc=fsmon}
The functor $\sigma$ induces an equivalence between the category of toric monoids and the category of rational polyhedral cones. 
\end{proposition}

\begin{proof}
Consider the functor $(.)^\vee$ that sends $(\sigma, N)$ in $\mathbf{RPC}$ to $S_\sigma=\sigma^{\vee}\cap M$, where $M=\Hom(N,\Z)$ and 
\begin{equation*}
\sigma^\vee=\big\{v\in M_\R\big\vert \langle u,v\rangle\geq 0 \text{ for all } v\in\sigma \big\} \ . 
\end{equation*}
By Gordon's Lemma \cite[Section 1 Proposition 1]{Fulton_toricvarieties} the monoid $S_\sigma$ is finitely generated and it is immediate that $S_\sigma$ is integral, saturated, and torsion-free. A morphism $f\mathrel{\mathop:}(\sigma,N)\rightarrow(\sigma', N')$ induces a homomorphism $f^{\vee}\mathrel{\mathop:}S_{\sigma'}\rightarrow S_\sigma$ and this association is functorial in $f$. It is now easy to check that $(.)^\vee$ is an inverse to $\sigma$ using Lemma \ref{lemma_fsmondecomp} (i).
\end{proof}

By Lemma \ref{lemma_fsmondecomp} (ii) the category of sharp toric monoids corresponds to rational polyhedral cones $(\sigma,N)$ that are \emph{sharp}, i.e. those cones $\sigma$ whose span in $N_\R$ is equal to $N_\R$. Throughout this article we are going to assume that all of our cones are sharp. In a slight abuse of notation we are therefore going to denote the category of sharp cones by $\mathbf{RPC}$; an object in $\mathbf{RPC}$ will simply be referred to as a \emph{cone} and written as $\sigma$ without explicit reference to the lattice $N$. 

A \emph{face morphism} $\tau\rightarrow\sigma$ is a morphism of cones that induces an isomorphism onto a (not necessarily proper) face of $\sigma$. Note that we explicitly allow automorphisms of a cone $\sigma$ in the class of face morphism. If we want to $\tau$ to be isomorphic to a proper face of $\sigma$, we refer to $\tau\rightarrow\sigma$ as a \emph{proper face morphism}. 
 
\subsection{Monoidal spaces}\label{section_monoidalspaces}

A \emph{locally monoidal space} is a pair $(X,\mathcal{O}_X)$ consisting of a topological space $X$ together with a sheaf of monoids $\mathcal{O}_X$. Given two locally monoidal spaces $(X,\mathcal{O}_X)$ and $(Y,\mathcal{O}_Y)$, a morphism of  locally monoidal spaces is a continuous map $f\mathrel{\mathop:}X\rightarrow Y$ together with a morphism $f^\dagger\mathrel{\mathop:}f^\ast\mathcal{O}_Y\rightarrow \mathcal{O}_X$  of sheaves of monoids such that the induced homomorphism $f_x^\dagger\mathrel{\mathop:}\mathcal{O}_{Y,f(x)}\rightarrow \mathcal{O}_{X,x}$ is a local homomorphism of monoids for all $x\in X$. This means that $f_x^\dagger(\mathfrak{m}_{Y,f(x)})\subseteq \mathfrak{m}_{X,x}$ for the unique maximal ideals $\mathfrak{m}_{Y,f(x)}$ and $\mathfrak{m}_{X,x}$ in $ \mathcal{O}_{X,x}$ and $\mathcal{O}_{Y,y}$ respectively. 

Denote the category of locally monoidal spaces by $\mathbf{LMS}$. A morphism $f\mathrel{\mathop:}X\rightarrow Y$ in $\mathbf{LMS}$ is said to be \emph{strict}, if the induced morphism $f^{-1}\mathcal{O}_Y\rightarrow\mathcal{O}_X$ is an isomorphism of sheaves of monoids on $X$. All schemes are implicitly thought of as monoidal spaces with respect to multiplication on $\mathcal{O}_X$. A monoidal space $(X,\mathcal{O}_X)$ is said to be \emph{sharp}, if $\mathcal{O}_{X,x}^\ast = 0$ for all $x\in X$. The category $\mathbf{SMS}$ of sharp monoidal spaces is a full and faithful subcategory of $\mathbf{LMS}$ and the association $(X,\mathcal{O}_X)\mapsto\overline{X}=(X,\overline{\mathcal{O}}_X)$ with $\overline{\mathcal{O}}_X=\mathcal{O}_X/\mathcal{O}_X^\ast$ defines a retraction functor onto this subcategory. 

%%%%%%%%%%%%%%%%%%%%%%%%%%%%%%%%%%%%%%%%%%%%%%%%%%%%%%

\section{Kato fans, cone complexes, and their extensions}\label{section_Katofans&conecomplexes}

\subsection{Kato fans}\label{section_Katofans}

In \cite[Section 9]{Kato_toricsing} K. Kato introduces the notion of a \emph{fan} that serves as a geometric model for the dual category $\mathbf{sh-Mon}^{op}$ of the category $\mathbf{sh-Mon}$ of sharp monoids (also see \cite[Section 3.5]{GabberRomero_foundationsalmostring} for further details). These objects should be thought of as analogues of schemes, where, instead of rings, we allow monoids as the fundamental building blocks.  

In particular, there is a functor
\begin{equation*}
\Spec\mathrel{\mathop:}\mathbf{Mon}^{op}\longrightarrow\mathbf{SMS}
\end{equation*}
that associates to a monoid $P$ a sharp monoidal space $\Spec P$, called the \emph{spectrum} of $P$. By \cite[Proposition 9.2]{Kato_toricsing} the spectrum $\Spec P$ is uniquely determined by representing the functor $\mathbf{SMS}\rightarrow\mathbf{Sets}$ that associates to a \emph{sharp} monoidal space $(X,\mathcal{O}_X)$ the set of homomorphisms $\Hom\big(P,\mathcal{O}_X(X)\big)$. 

As a set $\Spec P$ is equal to set of prime ideals of $P$ and its topology is the one generated by the open sets $D(f)=\{\mathfrak{p}\in\Spec P\vert f\notin\mathfrak{p}\}$ for $f\in P$. The structure sheaf $\mathcal{O}_F$ on $F=\Spec P$ is determined by the association
\begin{equation*}
D(f)\longmapsto P_f/P_f^\ast
\end{equation*}
and, consequently, the stalk of $\mathcal{O}_F$ at $\mathfrak{p}\in\Spec P$ is given by
\begin{equation*}
\mathcal{O}_{F,\mathfrak{p}}=P_\mathfrak{p}/P_\mathfrak{p}^\ast \ .
\end{equation*}
A morphism $\phi\mathrel{\mathop:}Q\rightarrow P$, induces a morphism $\phi^\#\mathrel{\mathop:}\Spec P\rightarrow\Spec Q$ that is given by the association $\mathfrak{p}\mapsto\phi^{-1}(\mathfrak{p})$ and the induced morphisms on the structure sheaves.

\begin{remark} \label{remark_torusorbits}A toric monoid $P$ defines a toric variety $X_P=\Spec k[P]$. The affine Kato fan $\Spec P$ is naturally homeomorphic to $\Xi(X_P)$, the set of generic points of the $T=\Spec k[P^{gp}]$-orbits in $X_P$. 
\end{remark}

\begin{proposition}\label{prop_afffans=shmon}
The functor $\Spec$ defines an equivalence between the category of sharp monoids and the category of affine Kato fans.
\end{proposition}

\begin{proof}
Given a sharp monoid $P$, we have $\mathcal{O}_{\Spec P}(\Spec P)=P/P^\ast=P$ and, conversely, for an affine Kato fan $F$ the identity $\Spec \mathcal{O}_F(F)=F$. So taking global sections of the structure sheaf defines an inverse to $\Spec$. 
\end{proof}

Note that the specialization relation defines a partial order on $\Spec P$ with a unique minimal element $\emptyset$ and a unique maximal element $\mathfrak{m}_P=P-P^\ast$. In the following pictures we indicate specialization by an arrow. 

\begin{examples}\begin{enumerate}[(i)]
\item If $P=\R_{\geq 0}$, then $\Spec\R_{\geq 0}$ consists of the two prime ideals $\emptyset$ and $\R_{>0}$ in the monoid $\R_{\geq 0}$. 
\begin{center}\begin{tikzpicture}
\fill (0,0) circle (0.12 cm)
      (2,0) circle (0.05 cm);
\draw [->] (0.4,0) -- (1.7,0);
\node at (0,0.5) {$0$};
\node at (2,0.5) {$\R_{\geq 0}$};
\end{tikzpicture}\end{center}
In this picture we indicate the stalk of the structure sheaf at each point. If $P=\Rbar_{\geq 0}$, then $\Spec\Rbar_{\geq 0}$ contains, in addition to $\emptyset$ and $\Rbar_{>0}$, also the prime ideal $\{\infty\}$. 

\item Similarly, if $P=\N$, then $F_{\A^1}=\Spec \N$ consists of the two prime ideals $\emptyset$ and $\N_{> 0}$ in the monoid $\N$.
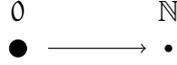
\begin{center}\begin{tikzpicture}
\fill (0,0) circle (0.12 cm)
      (2,0) circle (0.05 cm);
\draw [->] (0.4,0) -- (1.7,0);
\node at (0,0.5) {$0$};
\node at (2,0.5) {$\NN$};
\end{tikzpicture}\end{center}
Note that the topological spaces underlying $\Spec \N$ and $\Spec \R_{\geq 0}$ are the same; the crucial difference lies in their sheaf of functions.

\item Suppose $P=\N^2$. The spectrum $F_{\A^2}=\Spec \N^2$ consists of the four elements $\emptyset$, $\N\times\N_{>0}$, $\N_{>0}\times\N$, and $\N^2-\{0\}$. 
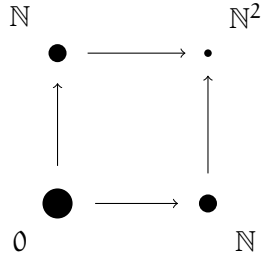
\begin{center}\begin{tikzpicture}
\fill (1,1) circle (0.20 cm)
      (3,1) circle (0.12 cm)
      (1,3) circle (0.12 cm)
      (3,3) circle (0.05 cm);
\draw [->] (1.5,1) -- (2.6,1);
\draw [->] (1.4,3) -- (2.7,3);
\draw [->] (1,1.5) -- (1,2.6);
\draw [->] (3,1.4) -- (3,2.7);

\node at (0.5,0.5) {$0$};
\node at (3.5,0.5) {$\NN$};
\node at (0.5,3.5) {$\NN$};
\node at (3.5,3.5) {$\NN^2$};
\end{tikzpicture}\end{center}
This visualization immediately generalizes to $P=\N^k$ for all positive integers $k$.
\item Let $P$ be the monoid generated by $p,q,r$ subject to the relation $p+r=2q$. Then $\Spec P$ has the following four points: $\emptyset$, $P-\N\cdot p$, $P-\N\cdot r$, and $P-\{0\}$.
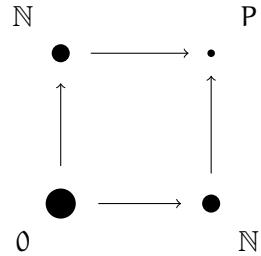
\begin{center}\begin{tikzpicture}
\fill (1,1) circle (0.20 cm)
      (3,1) circle (0.12 cm)
      (1,3) circle (0.12 cm)
      (3,3) circle (0.05 cm);
\draw [->] (1.5,1) -- (2.6,1);
\draw [->] (1.4,3) -- (2.7,3);
\draw [->] (1,1.5) -- (1,2.6);
\draw [->] (3,1.4) -- (3,2.7);
\node at (0.5,0.5) {$0$};
\node at (3.5,0.5) {$\NN$};
\node at (0.5,3.5) {$\NN$};
\node at (3.5,3.5) {$P$};
\end{tikzpicture}\end{center}
\end{enumerate}\end{examples}

In analogy with the category of schemes that extends the dual of the category of rings, K. Kato \cite{Kato_toricsing} introduces the category of \emph{Kato fans} the extends the dual of the category of sharp monoids. 

\begin{definition}
A \emph{Kato fan} $F$ is a sharp monoidal space that admits a covering by open subset $U_i$ isomorphic to $\Spec P_i$ for some monoids $P_i$.
\end{definition} 

In the remainder of this article, unless noted otherwise, we are going to assume that every Kato fan is \emph{locally fine and saturated}, i.e. that we may choose the $P_i$ as above to be fine and saturated. Denote the category of (locally fine and saturated) Kato fans by $\mathbf{Fans}$.

\begin{examples}\begin{enumerate}[(i)]
\item Given two copies $U_0$ and $U_1$ of $\Spec\NN$, we can glue them over the generic point $\{\emptyset\}$. This defines the Kato fan $F_{\PP^1}$. 
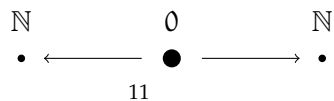
\begin{center}\begin{tikzpicture}
\fill (0,0) circle (0.12 cm)
      (2,0) circle (0.05 cm)
      (-2,0) circle (0.05 cm);
\draw [->] (0.4,0) -- (1.7,0);
\draw [->] (-0.4,0) -- (-1.7,0);

\node at (0,0.5) {$0$};
\node at (2,0.5) {$\NN$};
\node at (-2,0.5) {$\NN$};

\end{tikzpicture}\end{center}
\item Consider three copies $U_0$, $U_1$, and $U_2$ of $\Spec\NN^2$ with coordinates $p_0,q_0$, $p_1,q_1$, and $p_2,q_2$ respectively. Glue these affine fans with respect to the isomorphisms 
\begin{equation*}\begin{split}
D_{U_0}(p_0)&\simeq\Spec \NN\simeq D_{U_1}(q_1) \\
D_{U_1}(p_1)&\simeq\Spec\NN\simeq D_{U_2}(q_2) \\
D_{U_2}(p_2)&\simeq\Spec\NN\simeq D_{U_0}(q_0)
\end{split}\end{equation*}
in order to obtain the Kato fan $F_{\PP^2}$. 
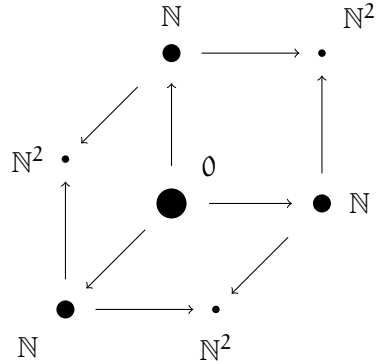
\begin{center}\begin{tikzpicture}
\fill (0,0) circle (0.20 cm)
      (2,0) circle (0.12 cm)
      (-1.41,-1.41) circle (0.12 cm)
      (0,2) circle (0.12 cm)
      (2,2) circle (0.05 cm)
      (0.59,-1.41) circle (0.05 cm)
      (-1.41,0.59) circle (0.05 cm);
      
\draw [->] (0.5,0) -- (1.6,0);
\draw [->] (0.4,2) -- (1.7,2);
\draw [->] (0,0.5) -- (0,1.6);
\draw [->] (2,0.4) -- (2,1.7);
\draw [->] (-1.41,-1.01) -- (-1.41,0.29);
\draw [->] (-1.01,-1.41) -- (0.29,-1.41);
\draw [->] (-0.35,-0.35) -- (-1.13,-1.13);
\draw [->] (-0.45, 1.55) -- (-1.2,0.8);
\draw [->] (1.55,-0.45) -- (0.8,-1.2);

\node at (0.5,0.5) {$0$};
\node at (2.5,0) {$\NN$};
\node at (0,2.5) {$\NN$};
\node at (2.5,2.5) {$\NN^2$};
\node at (-1.91,0.59) {$\NN^2$};
\node at (0.59,-1.91) {$\NN^2$};
\node at (-1.91,-1.91) {$\NN$};
\end{tikzpicture}\end{center}
Glueing $k+1$ copies of $\Spec \N^k$ in an analogous manner, gives rise to Kato fans $F_{\PP^k}$ for all integers $k$. 
\item Given four copies $U_1$, $U_2$, $U_3$, and $U_4$ of $\Spec\NN^2$ with generators $p_1,q_1$, $p_2,q_2$, $p_3,q_3$, and $p_4,q_4$ respectively, we can glue these affine Kato fans via the isomorphisms
\begin{equation*}\begin{split}
D_{U_1}(p_1)&\simeq\Spec \NN\simeq D_{U_2}(q_2)\\ 
D_{U_2}(p_2)&\simeq\Spec\NN\simeq D_{U_3}(q_3)\\
D_{U_3}(p_3)&\simeq\Spec\NN\simeq D_{U_4}(q_4)\\
D_{U_4}(p_4)&\simeq\Spec\NN\simeq D_{U_1}(q_1) \ . 
\end{split}\end{equation*}
The Kato fan obtained this way will be denoted by $F_{\PP^1\times\PP^1}$. 

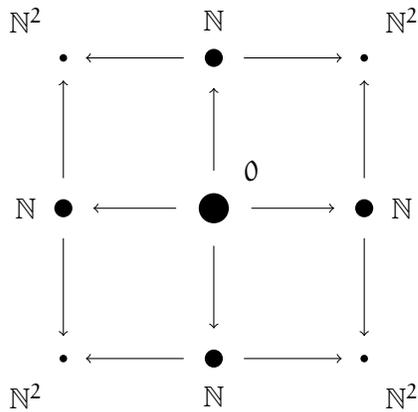
\begin{center}\begin{tikzpicture}
\fill (0,0) circle (0.20 cm)
      (2,0) circle (0.12 cm)
      (-2,0) circle (0.12 cm)
      (0,2) circle (0.12 cm)
      (0,-2) circle (0.12 cm)
      (2,2) circle (0.05 cm)
      (2,-2) circle (0.05 cm)
      (-2,2) circle (0.05 cm)
      (-2,-2) circle (0.05 cm);
\draw [->] (0.5,0) -- (1.6,0);
\draw [->] (0.4,2) -- (1.7,2);
\draw [->] (0,0.5) -- (0,1.6);
\draw [->] (2,0.4) -- (2,1.7);
\draw [->] (-0.5,0) -- (-1.6,0);
\draw [->] (-0.4,-2) -- (-1.7,-2);
\draw [->] (0,-0.5) -- (0,-1.6);
\draw [->] (-2,-0.4) -- (-2,-1.7);
\draw [->] (-2,0.4) -- (-2,1.7);
\draw [->] (2,-0.4) -- (2,-1.7);
\draw [->] (0.4,-2) -- (1.7,-2);
\draw [->] (-0.4,2) -- (-1.7,2);

\node at (0.5,0.5) {$0$};
\node at (2.5,0) {$\NN$};
\node at (0,2.5) {$\NN$};
\node at (2.5,2.5) {$\NN^2$};
\node at (-2.5,0) {$\NN$};
\node at (0,-2.5) {$\NN$};
\node at (-2.5,-2.5) {$\NN^2$};
\node at (2.5,-2.5) {$\NN^2$};
\node at (-2.5,2.5) {$\NN^2$};
\end{tikzpicture}\end{center}
An immediate generalization of this construction yields the Kato fans $F_{(\PP^1)^k}$ for all positive integers $k$.
\end{enumerate}
\end{examples}

The notation $F_{\A^n}$, $F_{\PP^n}$, and $F_{(\PP^1)^n}$ is explained in Example \ref{example_charmortoric} below: The Kato fans described here turn out to be the Kato fans that are naturally associated to the toric varieties $\A^n$, $\PP^n$, and $(\PP^1)^n$. Note, in particular, that the underlying topological spaces of these Kato fans precisely correspond to the generic points of the torus orbits of these toric varieties (see Remark \ref{remark_torusorbits} above).  

%%%%%%%%%%%%%%%%%%%%%%%%%%%%%%%%%%%%%%%%%%%%%%%%%%%%%%

\subsection{Cone complexes}\label{section_conecomplexes}
In \cite{Kato_toricsing} K. Kato introduced the notion of a Kato fan in order to algebraize much more geometric objects, so called \emph{rational polyhedral cone complexes} in the terminology of \cite[Section 2.1 Definition 5]{KKMSD_toroidal} (also see \cite[Section 2.1]{AbramovichCaporasoPayne_tropicalmoduli}).

\begin{definition}
A \emph{rational polyhedral cone complex} $\Sigma$ (or short: a \emph{cone complex}) consists of a topological space $\vert\Sigma\vert$ together with a collection of rational polyhedral cones $\sigma_\alpha$ and continuous maps $\phi_\alpha\mathrel{\mathop:}\sigma_\alpha\rightarrow \vert\Sigma\vert$ such that the following properties hold:
\begin{enumerate}[(i)]
\item The maps $\phi_\alpha$ are injective and induce a bijection 
\begin{equation*}
\bigsqcup_{\alpha}\mathring{\sigma}_\alpha \xrightarrow{\sim} \vert\Sigma\vert \ .
\end{equation*}
\item Given a proper face $\tau$ of $\sigma_\alpha$, then $\tau$ is also a member of the family $(\sigma_\alpha)$. 
\item A subset $A$ of $\vert\Sigma\vert$ is closed if and only if its preimages $\phi_\alpha^{-1}(A)$ are closed in $\sigma_\alpha$ for all $\alpha$.
\end{enumerate}
\end{definition}

Denote the category of cone complexes with piecewise $\Z$-linear morphisms by $\mathbf{RPCC}$. That is, a morphism $f\mathrel{\mathop:}\Sigma\rightarrow\Sigma'$ in this category is given by a continuous map $\vert f\vert\mathrel{\mathop:}\vert\Sigma\vert\rightarrow\vert\Sigma'\vert$ together with a family of morphisms $\sigma_\alpha\rightarrow \sigma'_\beta$ with $\beta=\beta(\alpha)$ such that the diagrams
\begin{equation*}\begin{CD}
\sigma_\alpha @>\phi_\alpha>>\vert\Sigma\vert\\
@VVV @VV\vert f\vert V\\
\sigma'_\beta @>\phi_\beta'>>\vert\Sigma'\vert
\end{CD}\end{equation*}
commute for all $\alpha$. 

\begin{proposition}\label{prop_conecomplex=fsfan}
There is an equivalence 
\begin{equation*}\begin{split}
\mathbf{Fans}&\xlongrightarrow{\sim}\mathbf{RPCC}\\
F&\longmapsto\Sigma_F 
\end{split}\end{equation*}
between the the category of Kato fans and the category of rational polyhedral cone complexes such that 
\begin{equation*}
\vert\Sigma_F\vert=F(\R_{\geq 0})=\Hom\big(\Spec\R_{\geq 0},F\big) \ .
\end{equation*}
\end{proposition}

Define the \emph{reduction map} $r\mathrel{\mathop:}\vert\Sigma_F\vert=F(\R_{\geq 0})\rightarrow F$ by sending a morphism $u\mathrel{\mathop:}\Spec\R_{\geq 0}\rightarrow F$ to the image $u(\R_{>0})$ of the maximal ideal $\R_{>0}$ in $\Spec\R_{\geq 0}$.

\begin{proof}[Proof of Proposition \ref{prop_conecomplex=fsfan}]
Let $F$ be a Kato fan. For an open affine subset $U=\Spec P$ in $F$, the preimage 
\begin{equation*}
\sigma_U=r^{-1}(U)=\Hom(P,\R_{\geq 0})
\end{equation*}
is a rational polyhedral cone in $\vert\Sigma_F\vert=F(\R_{\geq 0})$. Given an open affine subset $V=\Spec Q$ of $U$, we can assume that $Q=P_\mathfrak{p}$ for some prime ideal $\mathfrak{p}$ in $P$ and we have that 
\begin{equation*}
\sigma_V=r^{-1}(V)=\Hom(Q,\R_{\geq 0})
\end{equation*}
is a face of $\sigma_U$. This, in particular, allows us to endow $\vert\Sigma_F\vert$ with the weak topology. That is, a subset $A\subseteq\vert\Sigma_F\vert$ is closed, if and only if the $A\cap r^{-1}(U)\subseteq\Hom(P,\R_{\geq 0})$ is closed for all open affine subsets $U=\Spec P$ of $F$. These observations together with Proposition \ref{prop_rpc=fsmon} and Proposition \ref{prop_afffans=shmon} imply that $\mathbf{RPCC}$ and $\mathbf{Fans}$ are equivalent.
\end{proof}

\subsection{Canonical extensions}\label{section_extensions} Consider a Kato fan $F$. The results in \cite{Thuillier_toroidal} suggest to think of $\Sigma_F$ as the analogue of a non-Archimedean analytic space associated of $F$. In fact, the correct analogue of $X^\beth$ for a scheme $X$ locally of finite type over $k$ is not $\Sigma_F$, but rather its \emph{canonical compactification} $\Sigmabar_F$, as defined in \cite[Section 2.2]{AbramovichCaporasoPayne_tropicalmoduli}. 

\begin{definition}
Let $F$ be a Kato fan. The \emph{extended cone complex} associated to $F$ (also known as the \emph{canonical extension of $\Sigma_F$}) is the set $F(\Rbar_{\geq 0})$ of $\Rbar_{\geq 0}$-valued points on $F$, that is the set of morphisms $\Spec\Rbar_{\geq 0}\rightarrow F$. 
\end{definition}

Define the \emph{reduction map} 
\begin{equation*}
r\mathrel{\mathop:}\overline{\Sigma}_F\rightarrow F
\end{equation*} by sending a morphism $u\mathrel{\mathop:}\Spec\Rbar_{\geq 0}\rightarrow F$ onto the point $u(\Rbar_{>0})\in F$ and the \emph{structure map} 
\begin{equation*}
\rho\mathrel{\mathop:}\overline{\Sigma}_F\rightarrow F
\end{equation*} 
by sending a morphism $u\mathrel{\mathop:}\Spec\Rbar_{\geq 0}\rightarrow F$ to the point $u\big(\{\infty\}\big)\in F$. Note hereby that the reduction map $r\mathrel{\mathop:}\overline{\Sigma}_F\rightarrow F$ is a natural extension of the reduction map $r\mathrel{\mathop:}\Sigma_F\rightarrow F$ as defined right before the proof of Proposition \ref{prop_conecomplex=fsfan}. We observe that for an open affine subset $U=\Spec P$ of $F$ the preimage
\begin{equation*}
r^{-1}(U)=\Hom(P,\Rbar_{\geq 0})
\end{equation*}
is the canonical compactification $\sigmabar_U$ of the cone $\sigma_U=\Hom(P,\R_{\geq 0})$ as defined in \cite[Section 2]{Thuillier_toroidal}. Moreover, for an open affine subset $V=\Spec Q$ of $U$, the extended cone $\sigmabar_V=r^{-1}(V)$ is a face of $\sigmabar_U=r^{-1}(U)$ and so $\Sigmabar_F$ is the colimit of all $\sigmabar_U$ taken over all open affine subsets $U$ of $F$. This characterization of $\overline{\Sigma}_F$ allows us to endow it with the \emph{weak topology:} A subset $A\subseteq \overline{\Sigma}_F$ is closed if and only if $A\cap \sigmabar_U$ is closed for all open affine subsets $U=\Spec P$. 

\begin{proposition} The structure map $\rho\mathrel{\mathop:}\overline{\Sigma}_F\rightarrow F$ is continuous and the reduction map $r\mathrel{\mathop:}\overline{\Sigma}_F\rightarrow F$ is anti-continuous. 
\end{proposition}

\begin{proof}
We only need to consider an affine Kato fan $F=\Spec P$. Given $f\in P$, we have 
\begin{equation*}
\rho^{-1}\big(D(f)\big)=\big\{u\in\Hom(P,\overline{\mathbb{R}}_{\geq 0})\big\vert u(f)\neq\infty\big\}
\end{equation*}
as well as
\begin{equation*}
r^{-1}\big(D(f)\big)=\big\{u\in\Hom(P,\overline{\mathbb{R}}_{\geq 0})\big\vert u(f)=0\big\}
\end{equation*}
and this implies the continuity of $\rho$ as well as the anti-continuity of $r$. 
\end{proof}

\begin{proposition}\label{prop_KatofanECCfunc}
A morphism $f\mathrel{\mathop:}F\rightarrow G$ of Kato fans induces a continuous map 
\begin{equation*}\begin{split}
\Sigmabar(f)\mathrel{\mathop:}\overline{\Sigma}_F&\longrightarrow\overline{\Sigma}_G\\ 
u&\longmapsto f\circ u 
\end{split}\end{equation*}
that naturally extends the piecewise $\Z$-linear map $\Sigma(f)\mathrel{\mathop:}\Sigma_F\rightarrow \Sigma_F$ and makes the diagrams 
\begin{equation*}\begin{CD}
\overline{\Sigma}_F @>\Sigmabar(f)>>\overline{\Sigma}_G\\
@V\rho VV @VV\rho V\\
F@>f >>G
\end{CD}
\qquad\qquad
\begin{CD}
\overline{\Sigma}_F @>\Sigmabar(f)>>\overline{\Sigma}_G\\
@Vr VV @VVr V\\
F@>f >>G
\end{CD}\end{equation*}
commute. The association $f\mapsto \Sigmabar(f)$ is functorial in $f$. 
\end{proposition}

Note that on open affine subsets $U=\Spec P$ and $V=\Spec Q$ of $F$ and $G$ respectively with $f(U)\subseteq V$ the map $\Sigmabar(f)$ is given by 
\begin{equation*}\begin{split}
\Hom(P,\Rbar_{\geq 0})&\longrightarrow \Hom(Q,\Rbar_{\geq 0})\\
u&\longmapsto u\circ f^\# \ ,
\end{split}\end{equation*}
where $f^\#$ is a homomorphism $Q\rightarrow P$ inducing $f$. 

\begin{proof}
We only need to show the commutativity of the two diagrams. To achieve this we observe that for $u\in \Sigmabar_F=\Hom(\Spec\Rbar_{\geq 0})$ we have 
\begin{equation*}\begin{split}
\big(\rho_G\circ\Sigmabar(f)\big)(u)&=\rho_G (f\circ u)=\\
&=(f\circ u)(\{\infty\})=(f\circ \rho_F) (u)
\end{split}\end{equation*}
as well as
\begin{equation*}\begin{split}
\big(r_G\circ\Sigmabar(f)\big)(u)&=r_G (f\circ u)=\\
&=(f\circ u)(\Rbar_{>0})=(f\circ r_F) (u) \ .
\end{split}\end{equation*}
\end{proof} 

Recall from Section \ref{section_monoidalspaces} that a morphism $f\mathrel{\mathop:}F\rightarrow G$ of Kato fans is \emph{strict}, if the natural morphism $f^{-1}\mathcal{O}_G\rightarrow\mathcal{O}_F$ is an isomorphism of monoid sheaves on $F$ and note that this is equivalent to $f$ inducing isomorphisms $f_x^\dagger\mathrel{\mathop:}\mathcal{O}_{G,f(x)}\rightarrow\mathcal{O}_{F,x}$ for all $x\in F$.

\begin{corollary}
A morphism $f\mathrel{\mathop:}F\rightarrow G$ of Kato fans is strict if and only if the induced morphism $\Sigma(f)\mathrel{\mathop:}\Sigma_F\rightarrow\Sigma_G$ maps every cone in $\Sigma_F$ isomorphically onto a cone in $\Sigma_G$. 
\end{corollary}

\begin{proof} 
Denote for $x\in F$ and $y=f(x)\in G$ the unique minimal affine open subsets in $F$ and $G$ containing $x$ and $y$ respectively as their maximal points by $U=\Spec P$ and $V=\Spec Q$. A morphism $f\mathrel{\mathop:}F\rightarrow G$ is strict, if and only if it induces an isomorphism $U\simeq V$ for all points $x\in F$. By Proposition \ref{prop_KatofanECCfunc} this is the case if and only $f$ also induces isomorphisms 
\begin{equation*}
\Sigma_U=\Hom(P,\R_{\geq 0})\simeq\Sigma_V=\Hom(Q,\R_{\geq 0})
\end{equation*}
of rational polyhedral cones.% and
%\begin{equation*}
%\Sigmabar_U=\Hom(P,\Rbar_{\geq 0})\simeq\Sigmabar_V=\Hom(Q,\Rbar_{\geq 0})
%\end{equation*}
%of canonically compactified rational polyhedral cones. 
\end{proof}

\subsection{Stratification of extended cone complexes}
Let $F$ be a Kato fan. The collection $\rho^{-1}(x)$ for $x\in F$ defines a stratification of $\overline{\Sigma}_F$ by locally closed subsets. If $F$ is irreducible, the unique open stratum $\rho^{-1}(\eta)$, where $\eta$ is the generic point of $F$, is equal to the set of morphisms $\Spec\R_{\geq 0}\rightarrow F$, i.e. the cone complex $\Sigma_F$. In this sense we can think of $\Sigmabar_F$ as the \emph{canonical compactification} of $\Sigma_F$.

If $F=\Spec P$ is affine, we can formally describe this stratification following \cite[Section 3]{Payne_anallimittrop}  or \cite[Proposition 3.4]{Rabinoff_newtonpolygon}. Write $\sigma$ for the cone $\Hom(P,\mathbb{R}_{\geq 0})$ as well as $N_\R(\sigma)=\Hom(P,\Rbar)$. 

\begin{proposition}[\cite{Rabinoff_newtonpolygon} Proposition 3.4]\label{prop_troptoricstratification} There is a natural stratification
\begin{equation*}
\bigsqcup_{\tau\prec\sigma} N_\R/\Spec\tau \xlongrightarrow{\sim} N_\R(\sigma)
\end{equation*}
of $N_\R(\sigma)$ by locally closed subsets homeomorphic to $N_\R/\Span\tau$, given by associating to $[u]\in N_\R/\Span\tau$, represented by $u\in N_\R=\Hom(P^{gp},\R)$, the homomorphism
\begin{equation*}
p\longmapsto\left\{
  \begin{array}{l l}
   \langle u,p\rangle & \quad \text{if } p\in \tau^{\bot}\cap P\\
   \infty & \quad \text{else}\\
  \end{array} \right .
\end{equation*}
in $N_\R(\sigma)=\Hom(P,\Rbar)$. 
\end{proposition}

We may identify the stratification of $\overline{\Sigma}_F=\Hom(P,\overline{\mathbb{R}}_{\geq 0})\subseteq\Hom(P,\Rbar)$ given by the preimages $\rho^{-1}(x)$ of the points $x$ in $F$ with the stratification 
\begin{equation*}
\bigsqcup_{\tau\prec\sigma} \sigma/\tau \xlongrightarrow{\sim} \Sigmabar_F
\end{equation*}
induced from Proposition \ref{prop_troptoricstratification}, where we write $\sigma/\tau$ for the image of $\sigma$ in $N_\R/\Span\tau$. Here $\sigma/\tau$ is identified with a locally closed subset of $\overline{\Sigma}_F$ by sending an element in $\sigma/\tau$ represented by $u\in\sigma$ to the homomorphism 
\begin{equation*}
p\longmapsto\left\{
  \begin{array}{l l}
   \langle u,p\rangle & \quad \text{if } p\in \tau^{\bot}\cap P\\
   \infty & \quad \text{else}\\
  \end{array} \right .
\end{equation*}
in $\Hom(P,\overline{\mathbb{R}}_{\geq 0})$.

\begin{remark}
If we think of $N_\R(\sigma)=\Hom(P,\Rbar)$ as the analogue of $X^{an}$ for an affine scheme $X=\Spec A$  of finite type over $k$, then $\Sigmabar_F=\Hom(P,\Rbar_{\geq 0})$ is the closed subset corresponding to $X^\beth\subseteq X^{an}$, the set of bounded seminorms on $A$ (see Section \ref{section_analytification} below).
\end{remark}
 
\begin{examples}\begin{enumerate}[(i)]

\item If $F=F_{\A^1}=\Spec \N$, then $\overline{\Sigma}_F=\overline{\Sigma}_{\A^1}=\Rbar_{\geq 0}$. The strata are given by $\R_{\geq 0}$ and $\{\infty\}$. 

\begin{center}\begin{tikzpicture}
\fill (0,0) circle (0.12cm);
\fill (2,0) circle (0.05cm);

\draw (0,0) -- (1,0);
\draw[dotted] (1,0) -- (2,0);

\end{tikzpicture}\end{center}

\item Suppose $F=F_{\A^2}=\Spec \N^2$. Then $\overline{\Sigma}_F=\overline{\Sigma}_{\A^2}=\Rbar_{\geq 0}^2$ and the strata are given by $\R_{\geq 0}^2$, $\R_{\geq 0}\times\{\infty\}$, $\{\infty\}\times\R_{\geq 0}$, and $\{(\infty,\infty)\}$.
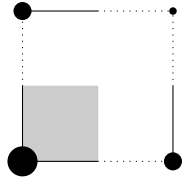
\begin{center}\begin{tikzpicture}

\draw [black!20!white, fill=black!20!white] (0,0) rectangle (1,1);

\fill (2,2) circle (0.05 cm);
\fill (0,0) circle (0.20 cm);
\fill (0,2) circle (0.12 cm);
\fill (2,0) circle (0.12 cm);

\draw (0,0) -- (1,0);
\draw (0,0) -- (0,1);
\draw (2,0) -- (2,1);
\draw (0,2) -- (1,2);
 
\draw [dotted] (0,1) -- (0,2);
\draw [dotted] (1,0) -- (2,0);
\draw [dotted] (2,1) -- (2,2);
\draw [dotted] (1,2) -- (2,2);

\end{tikzpicture}\end{center}

\item Suppose $P$ is the monoid generated by $p,q,r$ subject to the relation $p+r=2q$. Then the extended cone complex $\overline{\Sigma}_F$ associated to $F=\Spec P$ consists of the cone $\R_{\geq 0}(1,2)+\R_{\geq 0}(1,0)$, two copies of $\R_{\geq 0}$, and the point $\{(\infty,\infty)\}$ at infinity. 
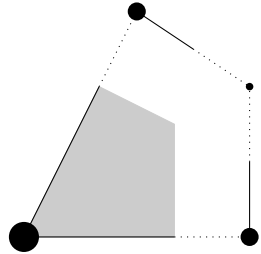
\begin{center}\begin{tikzpicture}

\draw [black!20!white, fill=black!20!white] (0,0) -- (1,2) -- (2,1.5) -- (2,0) -- (0,0);

\fill (3,2) circle (0.05 cm);
\fill (0,0) circle (0.20 cm);
\fill (1.5,3) circle (0.12 cm);
\fill (3,0) circle (0.12 cm);

\draw (0,0) -- (2,0);
\draw (0,0) -- (1,2);
\draw (1.5,3) -- (2.25,2.5);
\draw (3,0) -- (3,1);
 
\draw [dotted] (1,2) -- (1.5,3);
\draw [dotted] (2,0) -- (3,0);
\draw [dotted] (3,1) -- (3,2);
\draw [dotted] (2.25,2.5) -- (3,2);

\end{tikzpicture}\end{center}
\end{enumerate}\end{examples}

\begin{examples}\begin{enumerate}[(i)]
\item Let $F=F_{\PP^1}$. Then the strata of the extended cone complex are given by $\R$, $\{\infty\}$, and $\{-\infty\}$.
\begin{center}\begin{tikzpicture}
\fill (0,0) circle (0.12cm);
\fill (2,0) circle (0.05cm);
\fill (-2,0) circle (0.05cm);

\draw (0,0) -- (1,0);
\draw (0,0) -- (-1,0);

\draw[dotted] (1,0) -- (2,0);
\draw[dotted] (-1,0) -- (-2,0);

\end{tikzpicture}\end{center}

\item If $F=F_{\PP^2}$, its associated extended cone complex $\overline{\Sigma}_F=\overline{\Sigma}_{\PP^2}$ consists of the open stratum $\R^2$, three copies of $\R$ at infinity and three points $\{\infty\}$ at infinity. 

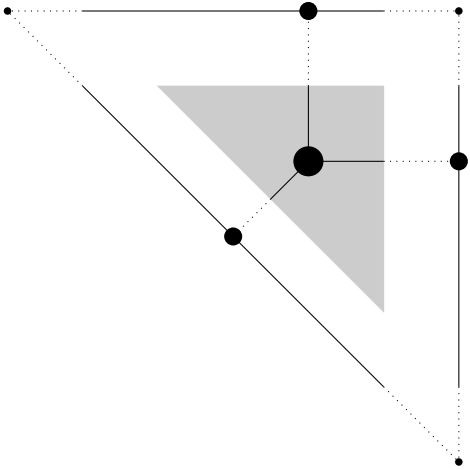
\begin{center}\begin{tikzpicture}
\draw [black!20!white, fill=black!20!white] (1,1) -- (-2,1) -- (1,-2) -- (1,1);

\fill (2,2) circle (0.05 cm);
\fill (0,0) circle (0.20 cm);
\fill (0,2) circle (0.12 cm);
\fill (-1,-1) circle (0.12 cm);
\fill (-4,2) circle (0.05 cm);
\fill (2,-4) circle (0.05 cm);
\fill (2,0) circle (0.12 cm);

\draw (0,0) -- (1,0);
\draw (0,0) -- (0,1);
\draw (2,0) -- (2,1);
\draw (0,2) -- (1,2);
\draw (0,0) -- (-0.5,-0.5);
\draw (0,2) -- (-3,2);
\draw (2,0) -- (2,-3);
\draw (-1,-1) -- (1,-3);
\draw (-1,-1) -- (-3,1);

\draw [dotted] (0,1) -- (0,2);
\draw [dotted] (1,0) -- (2,0);
\draw [dotted] (2,1) -- (2,2);
\draw [dotted] (1,2) -- (2,2);
\draw [dotted] (-0.5,-0.5) -- (-1,-1);
\draw [dotted] (-3,2) -- (-4,2);
\draw [dotted] (-3,1) -- (-4,2);
\draw [dotted] (1,-3) -- (2,-4);
\draw [dotted] (2,-3) -- (2,-4);

\end{tikzpicture}\end{center}

\item If $F=F_{\PP^1\times\PP^1}$, its associated extended cone complex $\overline{\Sigma}_F=\overline{\Sigma}_{\PP^1\times\PP^1}$ can be visualized as follows. 

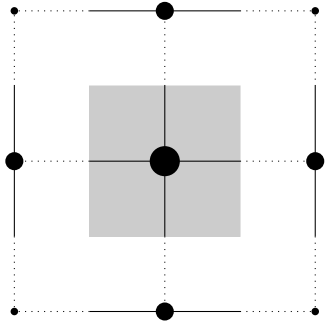
\begin{center}\begin{tikzpicture}

\draw [black!20!white, fill=black!20!white] (-1,-1) rectangle (1,1);

\fill (2,2) circle (0.05 cm);
\fill (0,0) circle (0.20 cm);
\fill (0,2) circle (0.12 cm);
\fill (2,0) circle (0.12 cm);
\fill (-2,-2) circle (0.05 cm);
\fill (0,-2) circle (0.12 cm);
\fill (-2,0) circle (0.12 cm);
\fill (2,-2) circle (0.05 cm);
\fill (-2,2) circle (0.05 cm);

\draw (0,0) -- (1,0);
\draw (0,0) -- (0,1);
\draw (2,0) -- (2,1);
\draw (0,2) -- (1,2);
\draw (0,0) -- (-1,0);
\draw (0,0) -- (0,-1);
\draw (-2,0) -- (-2,-1);
\draw (0,-2) -- (-1,-2);
\draw (-2,0) -- (-2,1);
\draw (0,-2) -- (1,-2);
\draw (2,0) -- (2,-1);
\draw (0,2) -- (-1,2);
 
\draw [dotted] (0,1) -- (0,2);
\draw [dotted] (1,0) -- (2,0);
\draw [dotted] (2,1) -- (2,2);
\draw [dotted] (1,2) -- (2,2);
\draw [dotted] (0,-1) -- (0,-2);
\draw [dotted] (-1,0) -- (-2,0);
\draw [dotted] (-2,-1) -- (-2,-2);
\draw [dotted] (-1,-2) -- (-2,-2);
\draw [dotted] (-2,1) -- (-2,2);
\draw [dotted] (-1,2) -- (-2,2);
\draw [dotted] (2,-1) -- (2,-2);
\draw [dotted] (1,-2) -- (2,-2);

\end{tikzpicture}\end{center}
The strata of $\overline{\Sigma}_{\PP^1\times\PP^1}$ are given by $\R^2$, $\{\infty\}\times \R$, $\R\times\{\infty\}$, $\{-\infty\}\times \R$, $\R\times\{-\infty\}$, $\big\{(\infty,\infty)\big\}$, $\big\{(\infty,-\infty)\big\}$, $\big\{(-\infty,\infty)\big\}$, and $\big\{(-\infty,-\infty)\big\}$.
\end{enumerate}\end{examples}

%%%%%%%%%%%%%%%%%%%%%%%%%%%%%%%%%%%%%%%%%%%%%%%%%%%%%%

\subsection{Generalized cone complexes}

In \cite{AbramovichCaporasoPayne_tropicalmoduli} the authors develop the notion of a \emph{generalized cone complex} in order to describe the combinatorial structure of a toroidal embedding that has self-intersection. Recall that a \emph{face morphism} $\tau\rightarrow\sigma$ is a morphism of rational polyhedral cones that induces an isomorphism onto a face of $\sigma$. This face need not be a proper face of $\sigma$; so, in particular, automorphisms of cones are allowed. 

\begin{definition}[\cite{AbramovichCaporasoPayne_tropicalmoduli} Section 2.6]
A \emph{generalized cone complex} $\Sigma$ is a topological space $\vert\Sigma\vert$ together with a presentation as a colimit of a diagram of face morphisms in $\mathbf{RPC}$ such that 
\begin{itemize}
\item for a proper face $\tau$ of a cone $\sigma$ in $\Sigma$ the proper face morphism $\tau\hookrightarrow\sigma$ is also in $\Sigma$, and
\item for every automorphism of a cone $\sigma$ in $\Sigma$ that leaves a proper face $\tau$ of $\sigma$ invariant the induced automorphism of $\tau$ is also in $\Sigma$. 
\end{itemize}
\end{definition}

\begin{example}\label{example_conecomplex=generalized}
Let $F$ be a Kato fan. Then the cone complex $\Sigma_F=F(\R_{\geq 0})$ is a generalized cone complex, since we may take the diagram of all $\sigma_U=\Hom(P,\R_{\geq 0})$ taken over open affine subsets $U=\Spec P$ of $F$ connected by (proper) face morphisms $\sigma_V\hookrightarrow \sigma_U$ whenever $V=\Spec Q$ is an open affine subset of $U$. 
\end{example}

\begin{example}\label{example_finitegroupquotients}
Let $G$ be a finite group acting on a cone $\sigma$ by automorphisms. Then the quotient $\sigma/G$ is a generalized cone complex (see Figure \ref{figure_R2modZ2}).
\end{example}

A morphism $\Sigma\rightarrow\Sigma'$ of generalized cone complexes is a continuous map $\vert\Sigma\vert\rightarrow\vert\Sigma'\vert$ such that for every $\sigma$ in $\Sigma$, there is a cone $\sigma'$ in $\Sigma'$  such that the composition $\sigma\rightarrow\Sigma\rightarrow\Sigma'$ factors through a morphism $\sigma\rightarrow\sigma'$ of cones. 

Of course, we may form the canonical extension of a generalized cone complex $\Sigma$ by taking the colimit of the $\sigmabar_\alpha$ instead of the $\sigma_\alpha$ (see Figure \ref{figure_R2modZ2}). In this case a morphism $f\mathrel{\mathop:}\Sigma\rightarrow\Sigma'$ of generalized cone complexes canonically extends to a continuous map $\overline{f}\mathrel{\mathop:}\Sigmabar\rightarrow\Sigmabar'$. 

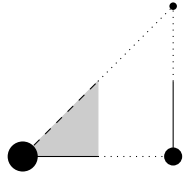
\begin{figure}[H]
\begin{center}\begin{tikzpicture}

\draw [black!20!white, fill=black!20!white] (0,0) -- (1,1) -- (1,0) -- (0,0);

\fill (2,2) circle (0.05 cm);
\fill (0,0) circle (0.20 cm);
\fill (2,0) circle (0.12 cm);

\draw (0,0) -- (1,0);
\draw[dashed] (0,0) -- (1,1);
\draw (2,0) -- (2,1);
 
\draw [dotted] (1,1) -- (2,2);
\draw [dotted] (1,0) -- (2,0);
\draw [dotted] (2,1) -- (2,2);

\end{tikzpicture}\end{center} \caption{The quotient of $\Rbar_{\geq 0}^2$ by the $\Z_2$-operation $(x,y)\mapsto (y,x)$ is a generalized extended cone complex by Example \ref{example_finitegroupquotients}.} \label{figure_R2modZ2}
\end{figure}

%%%%%%%%%%%%%%%%%%%%%%%%%%%%%%%%%%%%%%%%%%%%%%%%%%%%%%

%%%%%%%%%%%%%%%%%%%%%%%%%%%%%%%%%%%%%%%%%%%%%%%%%%%%%%

\section{Logarithmic structures and Kato fans}\label{section_logstr}

The main reference for logarithmic geometry is \cite{Kato_logstr}. We also refer the reader to \cite[Section 2-4]{Kato_logdef} and \cite[Section 2 and 3]{Abramovichetal_logmoduli} for short accounts of this theory and to \cite[Section 7]{GabberRomero_foundationsalmostring} for a treatment in full generality. 

\subsection{Logarithmic structures and charts}\label{section_logstructures}

One of the main objectives of logarithmic geometry in the sense of K. Kato \cite{Kato_logstr} is to enlarge the category of schemes in a way that allows us to generalize many of the nice combinatorial constructions that are naturally associated to toric varieties to more general schemes. The geometry of a toric variety $X$ with big torus $T$ is governed by the $T$-invariant divisors. If $X=\Spec k[P]$ is affine, the monoid of torus-invariant Cartier divisors is given by $P/P^\ast$. It is exactly this observation that motivates the following definition.

\begin{definition} Let $X$ be a scheme and denote by $X_\tau$ either the associated Zariski site $X_{Zar}$ or the \'etale site $X_{et}$.
\begin{enumerate}[(i)]
\item A \emph{pre-logarithmic structure} on $X_\tau$ is a pair $(M,\rho)$ consisting of a sheaf of monoids $M$ on $X_\tau$ and a morphism $\rho\mathrel{\mathop:}M\rightarrow\mathcal{O}_X$ of sheaves of monoids on $X$.
\item A pre-logarithmic structure $(M,\rho)$ is said to be a \emph{logarithmic structure}, if $\rho$ induces an isomorphism $\rho^{-1}\mathcal{O}_X^\ast\simeq\mathcal{O}_X^\ast$.
\end{enumerate}
\end{definition}

The triple $(X,M,\rho)$ is called a \emph{logarithmic scheme}. If the logarithmic structure is defined on $X_{Zar}$ (or on $X_{et}$), we say $(X,M,\rho)$ is a \emph{Zariski (resp. \'etale) logarithmic scheme}. It is very common to conveniently suppress the reference to $\rho$ or $(M,\rho)$ from this notation and to denote a logarithmic scheme just by $(X,M)$ or $X$ respectively. In the latter case the logarithmic structure will be written as $(M_X,\rho_X)$.

Denote by $\pi\mathrel{\mathop:}X_{et}\rightarrow X_{Zar}$ the natural morphisms of sites. As explained in \cite[Appendix A]{Olsson_loggeometryalgstacks} to every Zariski logarithmic structure $(M,\rho)$ we can associate its pullback $\pi^{\ast}(M,\rho)$ to $X_{et}$ and restricting an \'etale logarithmic structure $(M,\rho)$ from $X_{et}$ to $X_{Zar}$ defines an adjoint $\pi_\ast$ to $\pi^\ast$. By \cite[Theorem A.1]{Olsson_loggeometryalgstacks} the functor $\pi^\ast$ induces an equivalence between the category of Zariski logarithmic structures on $X_{Zar}$ and the category of \'etale logarithmic structures on $X_{et}$ for which the adjunction morphism $\pi^\ast\pi_\ast (M,\rho)\rightarrow (M,\rho)$ is an isomorphism. From now on we are going to identify a Zariski logarithmic structure $(M,\rho)$ on $X_{Zar}$ with its pullback to $X_{et}$.

\begin{example}[Divisorial logarithmic structures]\label{example_divlog}
Let $D$ be a divisor on a normal scheme $X$. Setting 
\begin{equation*}
M_D=\big\{f\in\mathcal{O}_X\big\vert f\vert_{X-D}\in\mathcal{O}_X^\ast\big\}
\end{equation*}
defines a logarithmic structure on $X$. 
\end{example}

\begin{example}[Toric varieties]
Let $X$ be a $T$-toric variety defined by a rational polyhedral fan $\Delta$ in the real vector space $N_\R=N\otimes \R$ associated to the cocharacter lattice $N$ of $T$. The divisorial logarithmic structure associated to the toric boundary $X-T$ is given by 
\begin{equation*}\begin{split}
k^\ast\oplus S_\sigma&\longrightarrow k[S_\sigma]\\
(a,s)&\longmapsto a\chi^s
\end{split}\end{equation*}
on $T$-invariant open affine subset $U_\sigma=\Spec K[S_\sigma]$ for a cone $\sigma$ in $\Delta$. 
\end{example}

To any pre-logarithmic structure $(M,\rho)$ on $X$ we can canonically associate a logarithmic structure $M^a$ together with a morphism $M\rightarrow M^a$ that is adjoint to the natural forgetful functor from the category of logarithmic structures to the category of pre-logarithmic structures on $X$. The sheaf $M^a$ is defined to be the pushout of 
\begin{equation*}\begin{CD}
\alpha^{-1}(\mathcal{O}_X^\ast)@>\subseteq>> M\\
@V\rho VV \\
\mathcal{O}_X^\ast
\end{CD}\end{equation*}
in the category of sheaves of monoids on $X$ with the induced morphism $M^a\rightarrow\mathcal{O}_X$. 

Given a morphism $f\mathrel{\mathop:}Y\rightarrow X$ of schemes, the \emph{inverse image} $f^\ast M$ of $M$ is defined to be the logarithmic structure associated to $f^{-1}M\rightarrow f^{-1}\mathcal{O}_X\rightarrow\mathcal{O}_Y$.

\begin{definition} \begin{enumerate}[(i)] 
\item Let $M$ and $M'$ be pre-logarithmic structures on $X$. A \emph{morphism of pre-logarithmic structures} $(M,\rho)\rightarrow (M',\rho')$ is a morphism $\gamma\mathrel{\mathop:}M\rightarrow M'$ of monoid sheaves such that $\rho'\circ \gamma=\rho$. A \emph{morphism of logarithmic structures} on $X$ is a morphism of pre-logarithmic structures. 
\item Let $X$ and $Y$ be logarithmic schemes. A \emph{morphism of logarithmic schemes} $f\mathrel{\mathop:}X\rightarrow Y$ consists of a morphism $f\mathrel{\mathop:}X\rightarrow Y$ of schemes together with a morphism $f^\flat\mathrel{\mathop:}f^\ast M_Y\rightarrow M_X$ of logarithmic structures on $X$. 
\end{enumerate}
\end{definition}

\begin{definition}
A \emph{chart} of a logarithmic structure $(M_X,\rho_X)$ on $X$ is given by a monoid $P$ and a morphism $\alpha\mathrel{\mathop:}P_X\rightarrow \calO_X$ such that $M_X$ is isomorphic to the logarithmic structure associated to $\alpha\mathrel{\mathop:}P_X\rightarrow \mathcal{O}_X$, where $P_X$ denotes the constant sheaf defined by $P$ on $X_\tau$.
\end{definition}

A morphism $f\mathrel{\mathop:}X\rightarrow Y$ of logarithmic schemes is said to be \emph{strict}, if $f^\flat\mathrel{\mathop:}f^\ast M_Y\rightarrow M_X$ is an isomorphism. A morphism $P_X\rightarrow \calO_X$ defines a chart if and only if the induced morphism $X\rightarrow\Spec\Z[P]$ is strict. 

A logarithmic scheme $X$ is said to be \emph{quasi-coherent}, if $X$ admits a covering by \'etale neighborhoods $U_i$ of $X$ such that for every $U_i$ there is a monoid $P_i$ and a chart $\beta_i\mathrel{\mathop:}(P_i)_{U_i}\rightarrow M_{U_i}$. If all $P_i$ can be chosen to be finitely generated (fine, saturated, etc.), we say the logarithmic structure $M$ is \emph{coherent (fine, saturated, etc.)}. 

Throughout this article, the term \emph{logarithmic scheme} will always mean a fine and saturated logarithmic scheme. Given two morphisms $X\rightarrow Z$ and $Y\rightarrow Z$ of logarithmic schemes, the fiber product  exists in the category of fine and saturated logarithmic schemes. If either (or both) of the morphisms are strict, then it is given  by endowing the scheme-theoretic fiber product $X\times_ZY$ with the logarithmic structure associated to $\pr_X^\#M_X\oplus_{M_Z}\pr_Y^\# M_Y$.

\subsection{Logarithmic schemes without monodromy}\label{section_monodromy}
The \emph{characteristic sheaf} of a logarithmic structure $(M_X,\rho_X)$ on $X$ is defined to be the sheaf $\overline{M}_X=M_X/M_X^\ast\simeq M_X/\mathcal{O}_X^\ast$ on $X$. We say a Zariski logarithmic scheme $X$ is \emph{small}, if the closed locus of points $x\in X$ where the restriction map
\begin{equation*}
\Gamma(X,\Mbar_X)\xlongrightarrow{\sim}\Mbar_{X,x}
\end{equation*}
is an isomorphism is non-empty and connected. 

\begin{definition}
A Zariski logarithmic scheme $X$ is said to have \emph{no monodromy}, if there is a strict morphism $(X,\Mbar_X)\rightarrow F$ into a Kato fan $F$.
\end{definition}

For a small Zariski logarithmic scheme $X$ the isomorphism $P_X=\Mbar_{X,x}\rightarrow\Gamma(X,\Mbar_X)$ induces a morphism
\begin{equation*}
\phi_X\mathrel{\mathop:}(X,\Mbar_X)\longrightarrow\Spec P_X
\end{equation*}
of sharp monoidal spaces that is strict, since $P_X\rightarrow \Mbar_X$ locally lifts to a chart of $M_X$ by \cite[Lemma (1.6)]{Kato_toricsing}. So a small logarithmic scheme automatically has no monodromy. 

\begin{proposition}\label{prop_charmor}
Suppose that $X$ is a Zariski logarithmic scheme without monodromy. There is a strict morphism $\phi_X\mathrel{\mathop:}(X,\overline{M}_X) \rightarrow F_X$ into a Kato fan $F_X$ that is initial among such morphisms.
\end{proposition}

In other words, given another strict morphism $\phi\mathrel{\mathop:}(X, \overline{M}_X)\rightarrow F$ into a Kato fan $F$, there is a unique strict morphism $F_X\rightarrow F$ making the diagram
\begin{center}\begin{tikzpicture}
  \matrix (m) [matrix of math nodes,row sep=1em,column sep=2em,minimum width=2em]
  {  
  (X,\overline{M}_X)& &F \\ 
  &  F_X &\\
  };
  \path[-stealth]
    (m-1-1) edge node [above] {$\phi$} (m-1-3)
    (m-1-1) edge node [below left] {$\phi_X$} (m-2-2)
    (m-2-2) edge node [right] {} (m-1-3);
\end{tikzpicture}\end{center}
commute. In a slight abuse of notation we also denote the composition $(X,\calObar_X)\rightarrow(X,\Mbar_X)\rightarrow F_X$ by $\phi_X$ and refer to it as the \emph{characteristic morphism} of $X$. 

\begin{proof}[Proof of Proposition \ref{prop_charmor}]
Suppose first that $X$ is small. Given a strict morphism $\phi\mathrel{\mathop:}(X,\overline{M}_X)\rightarrow F$ into a Kato fan $F$, the smallest open subset containing $\phi(x)$ as its unique closed point is isomorphic to $\Spec P_x$ therefore implying the universal property as above.

Now consider a general Zariski logarithmic scheme without monodromy. Choose a cover by small open subsets $U_i$. The intersections $U_{ij}=U_i\cap U_j$ may not be small, but we can again choose an open cover of $\bigsqcup_{ij}U_{ij}$ by small logarithmic schemes $V_k$. For every open $V_k\subseteq U_i$ the Kato fan $F_{V_k}$ is an open affine subset of $F_{U_i}$. Since there is a strict morphism $(X,\Mbar_X)\rightarrow F$ into some Kato fan, the diagram of $F_{V_k}$ and $F_{U_i}$ glues to give a Kato fan $F_X$, which has the desired universal property, since it is a colimit of all $U_i$ and $V_k$.
\end{proof}

Not every Zariski logarithmic scheme is without monodromy, as shown by the following Example \ref{example_GrossSiebertsexample} that has originally appeared in \cite[Appendix B]{GrossSiebert_logGromovWitten}.

\begin{example}[\cite{GrossSiebert_logGromovWitten} Example B.1]\label{example_GrossSiebertsexample}
Let $X$ be a logarithmic scheme whose underlying scheme is a union $C_1\cup C_2$ of two copies of $\PP^1$ meeting each other in two nodes $p$ and $q$. Set $U(p)=X-\{q\}$ and $U(q)=X-\{p\}$, and denote the two connected components of $U(p)\cap U(q)=X-\{p,q\}$ by $V_1$ and $V_2$ respectively. We define a logarithmic structure on $X$ as follows:
\begin{itemize}
\item On $U(p)=\Spec k[x_1,x_2]/(x_1x_2)$ we take the logarithmic structure $M_{U(p)}$ associated to the chart $\N^4\rightarrow k[x_1,x_2]/(x_1x_2)$ that is given by $(a_1,a_2,a_3,a_4)\mapsto x_1^{a_1}x_2^{a_2}$
\item Similarly, on $U(q)$ we consider the logarithmic structure $M_{U(q)}$ associated to the chart $\N^4\rightarrow k[y_1,y_2]/(y_1y_2)$ given by $(b_1,b_2,b_3,b_4)\mapsto y_1^{b_1}y_2^{b_2}$. 
\item Finally, on $V_1$ we glue the logarithmic structures along the identification $\N^3\simeq\N^3$ given by $(a_1,a_3,a_4)\simeq(b_1,b_3,b_4)$, while on $V_2$ we glue along the identification $\N^3\simeq\N^3$ given by $(a_2,a_3,a_4)\simeq(b_2,b_4,b_3)$. 
\end{itemize}
We obtain a diagram of four Kato fans $F_{U(p)}=\Spec\N^4$, $F_{U(q)}=\Spec\N^4$, $F_{V_1}=\Spec\N^3$, and $F_{V_2}=\Spec\N^3$. If we try to glue those four affine Kato fans, we find (going around the circle once) that an open affine subset $\Spec \N^2$ of $F_{U(p)}$ would have to be identified with itself along the swap $(a_3,a_4)\mapsto(a_4,a_3)$ and the resulting quotient is not a Kato fan. 
\end{example}

\begin{lemma}\label{lemma_charmorfunc}
A morphism $f\mathrel{\mathop:}X\rightarrow Y$ of logarithmic schemes $X$ and $Y$ locally of finite type over $k$ without monodromy induces a morphism $f^F\mathrel{\mathop:}F_X\rightarrow F_Y$ such that the diagram
\begin{equation*}\begin{CD}
(X,\overline{\mathcal{O}}_X)@>\phi_X>> F_X\\
@V(f,\overline{f}^\flat)VV @VV f^FV\\
(Y,\overline{\mathcal{O}}_Y)@>\phi_Y>> F_Y
\end{CD}\end{equation*}
is commutative. The association $f\mapsto f^F$ is functorial in $f$ and, if $f$ is a strict morphism, then $f^F$ is a strict morphism. 
\end{lemma}

\begin{proof}
The morphism $f$ induces a compatible family of morphisms $f^\flat_x\mathrel{\mathop:}\overline{M}_{Y,f(x)}\rightarrow\overline{M}_{X,x}$ for all $x\in X$ and therefore there is an induced morphism $f^F\mathrel{\mathop:}F_X\rightarrow F_Y$ that is functorial. Note, in particular, that a strict morphism $f\mathrel{\mathop:}X\rightarrow Y$ induces isomorphisms $\overline{M}_{Y,f(x)}\simeq\overline{M}_{X,x}$ for all $x\in X$ implying that $f^F$ is strict. 
\end{proof}

\subsection{Logarithmically smooth schemes}\label{section_logsmooth}

In \cite{Kato_logstr} K. Kato defines the notion of \emph{logarithmically smooth morphism} $X\rightarrow Y$ between logarithmic schemes and in \cite[Theorem 3.5]{Kato_logstr} (also see \cite[Theorem 4.1]{Kato_logdef}) he gives a criterion on how to check for logarithmic smoothness in terms of the existence of certain well-behaved charts. 

For $Y=\Spec k$, endowed with the trivial logarithmic structure $k^\ast\hookrightarrow k$, this reduces to the following: A logarithmic scheme $X$ is logarithmically smooth over $k$ if and only if every point $x\in X$ there is an \'etale neighborhood $\delta\mathrel{\mathop:}U\rightarrow X$, a fine and saturated monoid $P$ whose torsion part $P^{tors}$ has order invertible on $X$, as well as a chart $P_U\rightarrow \calO_U$ of $M_U=\delta^\ast M_X$ such that the induced morphism $\gamma\mathrel{\mathop:}U\rightarrow\Spec k[P]$ is \'etale. 

Write $X_0$ for the locus of points $x$ in $X$ where the logarithmic structure is trivial, i.e. where $\Mbar_{X,x}=\calO_{X,x}$. We have the following immediate Corollary \ref{cor_logsmooth=toroidal} of Kato's criterion.

\begin{corollary}\label{cor_logsmooth=toroidal}
Let $X$ be a logarithmic scheme that is logarithmically smooth over $k$. Then the open embedding $X_0\hookrightarrow X$ is a \emph{toroidal embedding}. 
\end{corollary}

In other words, $X$ is normal and for every point $x\in X$ there is an \'etale open neighborhood $\delta\mathrel{\mathop:}U\rightarrow X$ as well as an \'etale morphism $\gamma\mathrel{\mathop:}U\rightarrow Z$ into a toric variety $Z$ with big torus $T$ such that $\gamma^{-1}(T)=U_0=\delta^{-1}(X_0)$ (see \cite[Section 2.1]{Thuillier_toroidal}). Toroidal embeddings may also be defined using formal charts over an algebraic closure of the base field, as in \cite[Section 2.1]{KKMSD_toroidal}; both of these definitions are equivalent by \cite[Section 2]{Denef_toroidal}. 

\begin{proof}[Proof of Corollary \ref{cor_logsmooth=toroidal}]
By \cite[Proposition (8.3) and Theorem (4.1)]{Kato_toricsing} $X$ is normal. Choose a splitting $P=P^{tors}\oplus \tilde{P}$ with $\tilde{P}$ being toric as in Lemma \ref{lemma_fsmondecomp}. Since the order of $P^{tors}$ is invertible over $X$, the induced morphism $\Spec k[P]\rightarrow \Spec k[\tilde{P}]$ and therefore the composition $\gamma\mathrel{\mathop:}U\rightarrow \Spec k[P]\rightarrow\Spec k[\tilde{P}]$ is \'etale. Moreover, we also have that $\tilde{P}\rightarrow P\rightarrow \calO_U$ defines a chart of $M_U=\delta^{-1}M_X$, because $P^{tors}$ is always mapped into $\calO_U^\ast$. Therefore the morphism $\gamma\mathrel{\mathop:}U\rightarrow \Spec k[\tilde{P}]$ is strict and we have $\gamma^{-1}\big(\Spec k[\tilde{P}^{gp}]\big)=U_0$. Finally, since $\delta\mathrel{\mathop:}U\rightarrow X$ is a strict \'etale morphism we have $\delta^{-1}(X_0)=U_0$.  
\end{proof}

Suppose that $X$ is a logarithmic scheme that is logarithmically smooth over $k$. For a point $x\in X$ we denote by $\mathfrak{m}_{X,\xi}$ the unique maximal ideal in the local ring $\mathcal{O}_{X,x}$ and by $I(M,x)$ the ideal in $\mathcal{O}_{X,x}$ generated by $M_{X,x}-M_{X,x}^\ast$. Set
\begin{equation*}
\Xi(X)=\big\{\xi\in X\big\vert I(M_X,\xi)=\mathfrak{m}_{X,\xi}\big\} \ .
\end{equation*}

By \cite[Proposition (8.2) and (10.1)]{Kato_toricsing} we have that, if $M_X$ is defined in the Zariski topology, the sharp monoidal space $\big(\Xi(X),\Mbar_X\vert_{\Xi(X)}\big)$ is the Kato fan associated to $X$ and the characteristic morphism $\phi_X$ sends a point $x\in X$ to the point $\phi_X(x)$ that corresponds to the ideal $I(M,x)$ in $\mathcal{O}_{X,x}$. So, in particular, a Zariski logarithmic scheme $X$ that is logarithmically smooth over $k$ is without monodromy.

\begin{example}[Toric varieties]\label{example_charmortoric}
Let $X=X(\Delta)$ be a toric variety with big torus $T$. Then the characteristic fan $F_X$ can identified with the set of generic points of the $T$-orbits in $X$ and the structure sheaf is given by $\overline{M}_\Delta=M_\Delta/M_\Delta^\ast$. The characteristic morphism $\phi_\Delta\mathrel{\mathop:}(X,\overline{\mathcal{O}}_X)\rightarrow F_X$ is given by sending a point $x\in X$ to the generic point of the unique $T$-orbit containing $x$.
\end{example}

Let $X$ be an (\'etale) logarithmic scheme that is logarithmically smooth over $k$. The logarithmic structure on $X$ naturally induces a stratification of $X$ by locally closed subsets. For every $\xi\in\Xi(X)$ we have a stratum 
\begin{equation*}
E(\xi)=\overline{\{\xi\}}-\bigcup_{\xi\rightarrow\xi'}\overline{\{\xi'\}}
\end{equation*}
where the union on the right is running over all strict specializations of $\xi'$ of $\xi$ (written as $\xi\rightarrow\xi'$) with $\xi\neq \xi'$ and $\overline{\{.\}}$ denotes the closure of a point in $X$. Note, in particular, that, if $X$ is connected, then $X_0$ is the unique open stratum.

%%%%%%%%%%%%%%%%%%%%%%%%%%%%%%%%%%%%%%%%%%%%%%%%%%%%%%

\section{Non-Archimedean analytification -- a reminder}\label{section_analytification}

\subsection{Berkovich's analytification functor} Let $k$ be a field that is endowed with a possibly trivial non-Archimedean absolute value $\vert.\vert$. In \cite{Berkovich_book} Berkovich defines an analytification functor $(.)^{an}$ associating to a scheme $X$, locally of finite type over $k$, a locally ringed space $X^{an}$ that has desirable topological properties together with a structure morphism $\rho\mathrel{\mathop:}X^{an}\rightarrow X$. We refer to \cite[Section 2]{Gubler_guide} and \cite[Sections 2.2 and 3]{BakerPayneRabinoff_nonArchtrop} for down-to-earth treatments of these notions and to \cite{Berkovich_book} and \cite{Berkovich_etalecoho} for the general theory.

Recall that, if $X=\Spec(A)$ is affine, a point $x$ in $X$ is given by a multiplicative seminorm 
\begin{equation*}
\vert . \vert_x\mathrel{\mathop:}A\rightarrow\mathbb{R}_{\geq 0}
\end{equation*}
that extends the norm on $k$. In this case the topology on $X^{an}$ is the coarsest making the maps $X^{an}\rightarrow \mathbb{R}_{\geq 0}, \ x\mapsto\vert f\vert_x$ for all $f\in A$ continuous and $\rho$ is given by $x\mapsto\ker\vert.\vert_x$. In general, i.e. if $X$ is not affine, one can define the space $X^{an}$ by choosing a covering $U_i$ of $X$ by open and affine subsets and glueing the $U_i^{an}$ over all $(U_i\cap U_j)^{an}=\rho^{-1}(U_i\cap U_j)$.

Alternatively (see e.g. \cite[Section 2.5]{Payne_topnonArch}) we may characterize $X^{an}$ as the set of pair $(K, \phi)$ consisting of a non-Archimedean extension $K$ of $k$ and a morphism $\phi\mathrel{\mathop:}\Spec K\rightarrow X$ modulo an equivalence relation. Two such pairs $(K,\phi)$ and $(L,\psi)$ are equivalent, if there is a common non-Archimedean extension $\Omega$ of both $K$ and $L$ that makes the diagram
\begin{equation*}\begin{CD}
\Spec \Omega @>>> \Spec L\\
@VVV @VV\psi V\\
\Spec K @>\phi>> X
\end{CD}\end{equation*}
commute. 

\subsection{Thuillier's analytification functor} For the purpose of our construction it is more convenient to work with a slightly different analytification functor $(.)^{\beth}$, defined by Thuillier  \cite[Section 1]{Thuillier_toroidal} for all schemes $X$ that are locally of finite type over a trivially valued field $k$. This functor is closely related to the generic fiber functor in \cite{Berkovich_vanishingcyclesII} on the category of schemes that are locally of finite presentation over a valuation ring $R$. 

If $X=\Spec A$ is affine, the analytic space $X^\beth$ is the analytic domain $X^{\beth}$ of $X^{an}$ that is characterized by $\vert f\vert_x\leq 1$ for all $f\in A$. In the non-affine case choose an open affine covering $U_i$ of $X$ and glue the $U_i^\beth$ over $(U_i\cap U_j)^\beth=r^{-1}(U_i\cap U_j)$, where $r$ denotes the \emph{reduction map} $r\mathrel{\mathop:}X^\beth\rightarrow X$ that is induced by $x\mapsto\big\{f\in A\big\vert \vert f\vert_x<1\big\}$ on affine open subsets $U=\Spec(A)$. Similiarly to the above, every point in $X^\beth$ can be represented by a pair $(R,\phi)$ where $R$ is a valuation ring $R$ extending $k$ and $\phi\mathrel{\mathop:}\Spec R\rightarrow X$ is a morphism. 

The natural inclusion $(\Spec A)^\beth\subseteq (\Spec A)^{an}$ on the level of affine open patches induces a morphism $X^\beth\rightarrow X^{an}$ and, in a slight abuse of notation, we denote its composition with $\rho\mathrel{\mathop:}X^{an}\rightarrow X$ by $\rho$ as well. If $X$ is separated over $k$, the morphism $X^\beth\rightarrow X^{an}$ defines an isomorphism onto an analytic domain in $X^{an}$. If $X$ is proper over $k$, all $K$-rational points are already $R$-integral, where $R$ denotes the valuation ring of $K$, by the valuative criterion of properness. In this case the two functors $(.)^{an}$ and $(.)^\beth$ agree, that is we have $X^\beth=X^{an}$.

\subsection{Analytification of the projective line and its torus invariant subspaces} Let $k$ be a algebraically closed field carrying the trivial absolute value. In the following examples we illustrate the differences between the two analytification functors $(.)^{an}$ and $(.)^{\beth}$ using the affine line $\A^1$, the one-dimensional torus $\G_m$, and the projective line $\PP^1$. 

\begin{example}
The analytic space $(\A^1)^{an}$ consists of all seminorms on $k[x]$ extending the trivial absolute value on $k$. They can be classified as follows:
\begin{itemize}
\item  For every closed point $a$ in $\A^1$ with $a\in k$ we have the semi-norm $\vert .\vert_{a,0}$ that is determined by $\vert x-a\vert_{a,0} =0$. 
\item The \emph{Gauss point $\vert.\vert_\eta $} is determined by $\vert f\vert_\eta=1$ for all $f\in k[x]$.  
\item For every $a\in k$ and $0<r<1$ we have a seminorm $\vert.\vert_{a,r}$ that is uniquely determined by $\big\vert (x-a)\big\vert_{a,r}=r$.
\item Finally, for every $r>1$ there is a seminorm $\vert.\vert_{\infty,r}$ determined by $\vert x\vert =r$. \end{itemize}

We have $\lim_{r\rightarrow 1}\vert .\vert_{a,r}=\lim_{r\rightarrow 1}\vert .\vert_{\infty,r}=\vert.\vert_{\eta}$. If we associate $\vert.\vert_{a,0}$ to a closed point $a$ in $\A^1$ and the Gauss point $\vert.\vert_{1,0}$ to the generic point $\eta$ of $\A^1$ we can embed $\A^1$ into $(\A^1)^{an}$. So the analytic space $(\A^1)^{an}$ consists of compact intervals connecting a closed point $a$ in $\A^1$ to the Gauss point $\eta$ as well as an half-open interval that connects $\eta$ to $\infty$.

The analytic space $(\A^1)^\beth$ is the subset of $(\A^1)^{an}$ that does not contain all points of the form $\vert .\vert_{r}$ for $r>1$, since $(\A^1)^\beth$ only contains bounded seminorms. 
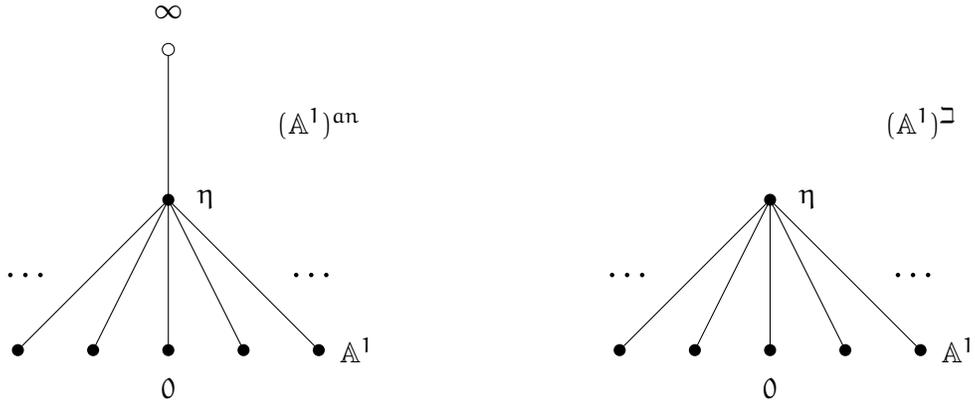
\begin{center}\begin{tikzpicture}

\draw (0,2) circle (0.08 cm);
\fill (0,0) circle (0.08 cm);
\fill (-2,-2) circle (0.08 cm);
\fill (-1,-2) circle (0.08 cm);
\fill (0,-2) circle (0.08 cm);
\fill (1,-2) circle (0.08 cm);
\fill (2,-2) circle (0.08 cm);

\draw (0,0) -- (0,1.92);
\draw (0,0) -- (-2,-2);
\draw (0,0) -- (-1,-2);
\draw (0,0) -- (0,-2);
\draw (0,0) -- (1,-2);
\draw (0,0) -- (2,-2);

\node at (0,-2.5) {$0$};
\node at (0,2.5) {$\infty$};
\node at (2.5,-2) {$\A^1$};
\node at (0.5,0) {$\eta$};
\node at (2,1) {$(\A^1)^{an}$};

\fill (-1.7,-1) circle (0.03cm);
\fill (-1.9,-1) circle (0.03cm);
\fill (-2.1,-1) circle (0.03cm);
\fill (1.7,-1) circle (0.03cm);
\fill (1.9,-1) circle (0.03cm);
\fill (2.1,-1) circle (0.03cm);

%\draw (8,2) circle (0.08 cm);
\fill (8,0) circle (0.08 cm);
\fill (6,-2) circle (0.08 cm);
\fill (7,-2) circle (0.08 cm);
\fill (8,-2) circle (0.08 cm);
\fill (9,-2) circle (0.08 cm);
\fill (10,-2) circle (0.08 cm);

%\draw (8,0) -- (8,1.92);
\draw (8,0) -- (6,-2);
\draw (8,0) -- (7,-2);
\draw (8,0) -- (8,-2);
\draw (8,0) -- (9,-2);
\draw (8,0) -- (10,-2);

\node at (8,-2.5) {$0$};
%\node at (6,2.5) {$\infty$};
\node at (10.5,-2) {$\A^1$};
\node at (8.5,0) {$\eta$};
\node at (10,1) {$(\A^1)^\beth$};

\fill (6.3,-1) circle (0.03cm);
\fill (6.1,-1) circle (0.03cm);
\fill (5.9,-1) circle (0.03cm);
\fill (9.7,-1) circle (0.03cm);
\fill (9.9,-1) circle (0.03cm);
\fill (10.1,-1) circle (0.03cm);

\end{tikzpicture}\end{center}
\end{example}

\begin{example} The analytic space $\G_m^{an}=\A^1-\{0\}$ is equal to $(\A^1)^{an}-\big\{\vert.\vert_{0,0}\big\}$ and henceforth contains two half-open intervals: One connects the Gauss point $\eta$ to $\infty$, the other connects $\eta$ to $0$. The analytic space $\G_m^\beth$, however, does not contain those two half-open intervals. One may think of $\G_m^\beth$ as a non-Archimedean analogue of $S^1$. 

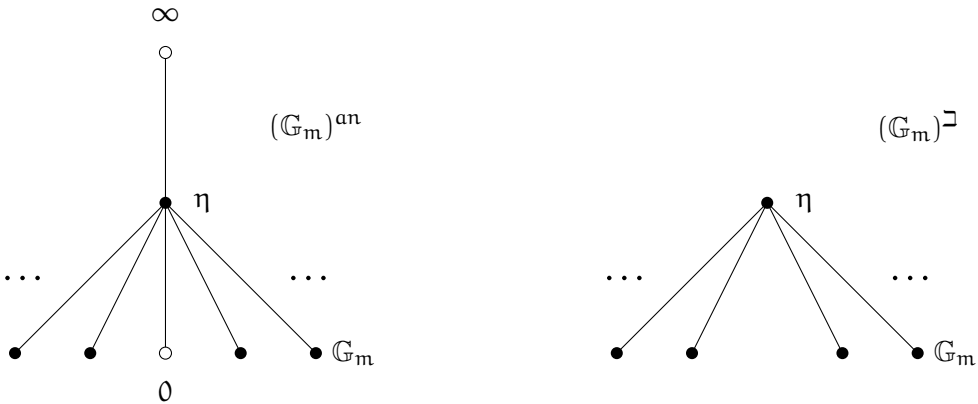
\begin{center}\begin{tikzpicture}

\draw (0,2) circle (0.08 cm);
\fill (0,0) circle (0.08 cm);
\fill (-2,-2) circle (0.08 cm);
\fill (-1,-2) circle (0.08 cm);
\draw (0,-2) circle (0.08 cm);
\fill (1,-2) circle (0.08 cm);
\fill (2,-2) circle (0.08 cm);

\draw (0,0) -- (0,1.92);
\draw (0,0) -- (-2,-2);
\draw (0,0) -- (-1,-2);
\draw (0,0) -- (0,-1.92);
\draw (0,0) -- (1,-2);
\draw (0,0) -- (2,-2);

\node at (0,-2.5) {$0$};
\node at (0,2.5) {$\infty$};
\node at (2.5,-2) {$\G_m$};
\node at (0.5,0) {$\eta$};
\node at (2,1) {$(\G_m)^{an}$};

\fill (-1.7,-1) circle (0.03cm);
\fill (-1.9,-1) circle (0.03cm);
\fill (-2.1,-1) circle (0.03cm);
\fill (1.7,-1) circle (0.03cm);
\fill (1.9,-1) circle (0.03cm);
\fill (2.1,-1) circle (0.03cm);

%\draw (8,2) circle (0.08 cm);
\fill (8,0) circle (0.08 cm);
\fill (6,-2) circle (0.08 cm);
\fill (7,-2) circle (0.08 cm);
%\fill (8,-2) circle (0.08 cm);
\fill (9,-2) circle (0.08 cm);
\fill (10,-2) circle (0.08 cm);

%\draw (8,0) -- (8,1.92);
\draw (8,0) -- (6,-2);
\draw (8,0) -- (7,-2);
%\draw (8,0) -- (8,-2);
\draw (8,0) -- (9,-2);
\draw (8,0) -- (10,-2);

%\node at (8,-2.5) {$0$};
%\node at (8,2.5) {$\infty$};
\node at (10.5,-2) {$\G_m$};
\node at (8.5,0) {$\eta$};
\node at (10,1) {$(\G_m)^\beth$};

\fill (6.3,-1) circle (0.03cm);
\fill (6.1,-1) circle (0.03cm);
\fill (5.9,-1) circle (0.03cm);
\fill (9.7,-1) circle (0.03cm);
\fill (9.9,-1) circle (0.03cm);
\fill (10.1,-1) circle (0.03cm);

\end{tikzpicture}\end{center}
\end{example}

\begin{example} The projective line $\PP^1$ is proper over $k$ and so $(\PP^1)^{an}$ and $(\PP^1)^\beth$ are equal. This space compactifies $(\A^1)^{an}$ by adding a point $\infty$ to the half-open interval connecting $\eta$ to $\infty$.
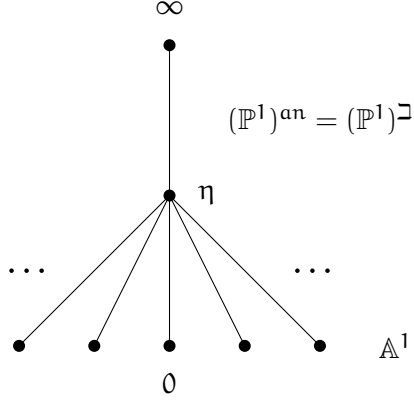
\begin{center}\begin{tikzpicture}

\fill (0,2) circle (0.08 cm);
\fill (0,0) circle (0.08 cm);
\fill (-2,-2) circle (0.08 cm);
\fill (-1,-2) circle (0.08 cm);
\fill (0,-2) circle (0.08 cm);
\fill (1,-2) circle (0.08 cm);
\fill (2,-2) circle (0.08 cm);

\draw (0,0) -- (0,2);
\draw (0,0) -- (-2,-2);
\draw (0,0) -- (-1,-2);
\draw (0,0) -- (0,-2);
\draw (0,0) -- (1,-2);
\draw (0,0) -- (2,-2);

\node at (0,-2.5) {$0$};
\node at (0,2.5) {$\infty$};
\node at (3,-2) {$\A^1$};
\node at (0.5,0) {$\eta$};
\node at (2,1) {$(\PP^1)^{an}=(\PP^1)^\beth$};

\fill (-1.7,-1) circle (0.03cm);
\fill (-1.9,-1) circle (0.03cm);
\fill (-2.1,-1) circle (0.03cm);
\fill (1.7,-1) circle (0.03cm);
\fill (1.9,-1) circle (0.03cm);
\fill (2.1,-1) circle (0.03cm);

\end{tikzpicture}\end{center}
\end{example}

Note that the structure morphism $\rho\mathrel{\mathop:}(\PP^1)^{an}\rightarrow\PP^1$ associates to a seminorm $\vert .\vert_{a,0}$ the closed point $a$ in $\PP^1$, to $\infty$ the point $\infty$ in $\PP^1$, but to all other points in $(\PP^1)^{an}$ the generic point $\eta$ of $\PP^1$. The reduction morphism $r\mathrel{\mathop:}(\PP^1)^\beth\rightarrow\PP^1$ associates only to the Gauss point $\vert.\vert_{1,0}$ the generic point of $\PP^1$. All other points $\vert. \vert_{a,r}$ for $r<1$ are mapped to the closed point $a$ in $\PP^1$, while the points $\vert.\vert_r$ for $r>1$ are mapped to $\infty$. It is exactly this dichotomy between $\rho$ and $r$ that lies at the very heart of many applications of tropical geometry.

%%%%%%%%%%%%%%%%%%%%%%%%%%%%%%%%%%%%%%%%%%%%%%%%%%%%%%

\section{Constructing the troplicalization map}\label{section_tropicalization}

%%%%%%%%%%%%%%%%%%%%%%%%%%%%%%%%%%%%%%%%%%%%%%%%%%%%%%

\subsection{Tropicalization via Kato fans -- the case without monodromy} Consider a (fine and saturated) Zariski logarithmic scheme $X$ locally of finite type over $k$ without monodromy and denote its characteristic morphism from Section \ref{section_monodromy} by $\phi_X\mathrel{\mathop:}(X,\overline{\mathcal{O}}_X)\rightarrow F_X$. As explained in Sections \ref{section_conecomplexes} and \ref{section_extensions}, we may associate to $X$ the \emph{cone complex} $\Sigma_X=\Sigma_{F_X}$ and its \emph{canonical extension} $\Sigmabar_X=\Sigmabar_{F_X}$.

\begin{definition}
The \emph{tropicalization map} associated to $X$
\begin{equation*}\begin{split}
\trop_X\mathrel{\mathop:}X^\beth&\longrightarrow\Sigmabar_X\\
x&\longmapsto \trop_X(x)
\end{split}\end{equation*}
is defined as follows: A point $x\in X^\beth$ can be represented by a morphism $\Spec R\rightarrow X$ for a valuation ring $R$ extending $k$ and this naturally induces a morphism $\underline{x}\mathrel{\mathop:}\Spec R\rightarrow (X,\overline{\mathcal{O}}_{X})$ in $\mathbf{SMS}$. Define the point $\trop_X(x)\in\Sigmabar_X=F_X(\Rbar_{\geq 0})$ as the composition
\begin{equation*}\begin{CD}
\Spec \Rbar_{\geq 0}@>\val^\#>> \Spec R@>\underline{x}>> (X,\overline{\mathcal{O}}_X) @>\phi_X>>X 
\end{CD}\end{equation*}
in the category $\mathbf{SMS}$, where $\val^\#$ denotes the morphism induced by the valuation $\val\mathrel{\mathop:}R\rightarrow \Rbar_{\geq 0}$ on $R$. 
\end{definition}

Note that $\trop_X$ is induced by $\phi_X$ in analogy with the analytic morphism $f^\beth\mathrel{\mathop:}X^\beth\rightarrow (X')^\beth$ being induced by a morphism $f\mathrel{\mathop:}X\rightarrow X'$ of schemes locally of finite over $k$.

\begin{proposition}\label{prop_tropprop}\begin{enumerate}[(i)]
\item The tropicalization map is well-defined and continuous. It makes the diagrams
\begin{equation*}\begin{CD}
X^\beth @>\trop_X>>\overline{\Sigma}_X\\
@Vr_XVV @VVr_{F_X}V\\
(X,\overline{\mathcal{O}}_X) @>\phi_X>> F_X
\end{CD}
\qquad\qquad
\begin{CD}
X^\beth @>\trop_X>>\overline{\Sigma}_X\\
@V\rho_X VV @VV\rho_{F_X} V\\
(X,\overline{\mathcal{O}}_X)@>\phi_X >>F_X
\end{CD}\end{equation*}
commute. 
\item A morphism $f\mathrel{\mathop:}X\rightarrow X'$ of Zariski logarithmic schemes locally of finite type over $k$ without monodromy induces a morphism $\Sigma(f)\mathrel{\mathop:}\Sigma_X\rightarrow\Sigma_{X'}$ whose canonical extension $\Sigmabar(f)$ makes the diagram
\begin{equation*}\begin{CD}
X^\beth @>\trop_X >>\Sigmabar_X\\
@Vf^\beth VV @VV\Sigmabar(f)V\\
(X')^\beth @>\trop_{X'}>>\Sigmabar_{X'}
\end{CD}\end{equation*}
commute. The association $f\mapsto\Sigma(f)$ is functorial in $f$. 
\end{enumerate}\end{proposition}

\begin{proof}
Given two representatives $\alpha\mathrel{\mathop:}\Spec R\rightarrow X$ and $\beta\mathrel{\mathop:}\Spec R'\rightarrow X$ of the same point $x\in X^\beth$, there is a valuation ring $\Omega$ extending both $R$ and $R'$ that makes the diagram 
\begin{equation*}\begin{CD}
\Spec \Omega @>>>\Spec R'\\
@VVV @VV\beta V\\
\Spec R @>\alpha>> X
\end{CD}\end{equation*}
commute. Thus we have $\val^\#\circ \alpha=(\val')^\#\circ\beta$ and so $\trop_X$ is well-defined. 

Suppose that both $X=\Spec A$ and $F_X=\Spec P$ are both affine and choose a homomorphism $\phi^\#\mathrel{\mathop:}P\rightarrow A$ inducing the characteristic morphism $\phi_X\mathrel{\mathop:}\Spec A\rightarrow\Spec P$. Then the tropicalization map 
\begin{equation*}
\trop_X\mathrel{\mathop:}X^\beth\longrightarrow\overline{\Sigma}_X=\Hom(P,\Rbar_{\geq 0})
\end{equation*} 
is given by 
\begin{equation*}
x\longmapsto \big(p\mapsto-\log\vert \phi^\#(p)\vert\big)
\end{equation*}
for $x\in X^\beth$. This shows that in the affine case $\trop_X$ is continuous; the general case follows by glueing, once we have established that the diagram for the reduction maps is commutative. 

The commutativity of the diagrams follows from the following characterization of $r_X$ and $\rho_X$: Given $x\in X^\beth$ represented by a morphism $\underline{x}\mathrel{\mathop:}\Spec R\rightarrow X$ for an integral valuation ring $R$ extending $k$, the reduction map $r_X$ sends $x$ to the point $\underline{x}(\eta_s)$ for the special point $\eta_s=\big\{a\in R\big\vert \vert a\vert < 1\big\}\in\Spec R$ and the structure morphism $\rho_X$ sends $x$ to the point $\underline{x}(\eta_g)$ for the generic point $\eta_g=\big\{a\in R\big\vert\vert a\vert =0\big\}\in\Spec R$. Therefore we have
\begin{equation*}\begin{split}
(\phi_X\circ r_X) (x) &=\phi_X\big( \underline{x}(\eta_s)\big)\\
&= (\phi_X\circ\underline{x}) \big(\val^{-1}(\Rbar_{>0})\big)\\
&= r_{F_X}\circ \val^\#\circ\phi_X\circ\underline{x}=(r_{F_X}\circ \trop_X)(x)
\end{split}\end{equation*}
as well as 
\begin{equation*}\begin{split}
(\phi_X\circ \rho_X) (x) &=\phi_X\big( \underline{x}(\eta_g)\big)\\
&= (\phi_X\circ\underline{x}) \big(\val^{-1}(\{\infty\})\big)\\
&= \rho_{F_X}\circ \val^\#\circ\phi_X\circ\underline{x}=(\rho_{F_X}\circ \trop_X)(x) \ .
\end{split}\end{equation*}

Part (ii) immediately follows from Lemma \ref{lemma_charmorfunc}, Proposition \ref{prop_KatofanECCfunc} and the definition of the tropicalization map. Note that we hereby set $\Sigma(f)=\Sigma(f^F)$ as well as $\Sigmabar(f)=\Sigmabar(f^F)$. 
\end{proof}

\subsection{Tropicalization in the presence of monodromy}

In this section we construct and describe the tropicalization map associated to a general \'etale logarithmic scheme $X$ that is locally of finite type over $k$. 

\begin{proposition}
There is a generalized cone complex $\Sigma_X$ and a continuous tropicalization map
\begin{equation*}
\trop_X\mathrel{\mathop:}X^\beth\longrightarrow\Sigmabar_X
\end{equation*}
into the canonical extension $\Sigmabar_X$ of $\Sigma_X$ such that for every strict surjective \'etale cover $U\rightarrow X$ by a logarithmic scheme $U$ without monodromy there is a commutative diagram
\begin{equation*}\label{eq_diagnomonodromy}\begin{CD}
U^\beth @>\trop_U>>\Sigmabar_U\\
@VVV @VVV\\
X^\beth @>\trop_X>>\Sigmabar_X
\end{CD}\end{equation*}
and the pair $(\Sigma_X,\trop_X)$ is initial among all such maps. 
\end{proposition}

In other words, given a continuous map $\tau\mathrel{\mathop:}X^\beth\rightarrow\Sigmabar$ into the canoncial extension of a generalized cone complex $\Sigma$ such that for all strict \'etale morphisms $U\rightarrow X$ from a logarithmic scheme without monodromy there is a morphism $\Sigma_U\rightarrow\Sigma$ making the diagram
\begin{equation*}\begin{CD}
U^\beth @>\trop_U>>\Sigmabar_U\\
@VVV @VVV\\
X^\beth @>\trop_X>>\Sigmabar
\end{CD}\end{equation*}
commute, then there is a unique morphism $\Sigma_X\rightarrow\Sigma$ making the diagram
\begin{center}\begin{tikzpicture}
  \matrix (m) [matrix of math nodes,row sep=1em,column sep=2em,minimum width=2em]
  {  
  X^\beth& &\Sigmabar \\ 
  &  \Sigmabar_X &\\
  };
  \path[-stealth]
    (m-1-1) edge node [above] {$\tau$} (m-1-3)
    (m-1-1) edge node [below left] {$\trop_X$} (m-2-2)
    (m-2-2) edge node [right] {} (m-1-3);
\end{tikzpicture}\end{center}
commute. 

\begin{proof}
For every strict \'etale morphism $U\rightarrow X$ from a small logarithmic scheme $U$ with $P_U=\Gamma(U,\Mbar_U)$, we have a cone $\sigma_U=\Hom(P_U,\R_{\geq 0})$ as well as a continuous tropicalization map 
\begin{equation*}
\trop_U\mathrel{\mathop:}U^\beth\longrightarrow\sigmabar_U \ .
\end{equation*}
Whenever there is another strict morphism $V\rightarrow U$ from a small logarithmic scheme $V$ to $U$, we have an induced face morphism $\sigma_V\rightarrow\sigma_U$. Define $\Sigma_X$ as the colimit of all $\sigma_U$, taken over all strict \'etale morphisms $U\rightarrow X$ from a small logarithmic scheme $U$.

By \cite[Lemma 6.1.3]{AbramovichCaporasoPayne_tropicalmoduli} the analytic space $X^\beth$ is the topological colimit of all such $U^\beth$ and therefore we may define a natural continuous tropicalization map 
\begin{equation*}
\trop_X\mathrel{\mathop:}X^\beth\longrightarrow\Sigmabar_X
\end{equation*}
by the universal property of colimits in the category of topological space. Since every logarithmic scheme $U$ without monodromy can be written as a colimit of small open subsets, the diagram \eqref{eq_diagnomonodromy} commutes as well.

Let $\tau\mathrel{\mathop:}X^\beth\rightarrow \Sigmabar$ be a continuous map into the canonical extension of a generalized cone complex $\Sigma$ as above. Then the diagram 
\begin{equation*}\begin{CD}
U^\beth @>\trop_Y>>\Sigmabar_U\\
@VVV @VVV\\
X^\beth @>\tau>>\Sigmabar
\end{CD}\end{equation*}
commutes for all small logarithmic schemes $U$ and therefore we obtain a unique morphism $\Sigma_X\rightarrow\Sigma$ as desired.
\end{proof}

\begin{proof}[Proof of Theorem \ref{thm_tropfunc}]
Let $f\mathrel{\mathop:}X\rightarrow X'$ be a morphism of logarithmic schemes. For every strict \'etale morphism $U'\rightarrow X'$ from a logarithmic scheme $U'$ without monodromy, the fiber product $U=X\times_{X'}U'$ is also without monodromy, since the composition $(U,\Mbar_U)\rightarrow (U',\Mbar_{U'})\rightarrow F_{U'}$ is strict. Therefore there is a unique morphism $\Sigma(f)\mathrel{\mathop:}\Sigma_X\rightarrow\Sigma_{X'}$ of generalized cone complexes that induces a natural continuous map $\Sigmabar(f)\mathrel{\mathop:}\Sigmabar_{X}\rightarrow\Sigmabar_{X'}$ making the diagram 
\begin{equation*}\begin{CD}
X^\beth @>\trop_X>>\overline{\Sigma}_X\\
@Vf^\beth VV @VV\Sigmabar(f) V\\
(X')^\beth @>\trop_{X'}>> \overline{\Sigma}_{X'} 
\end{CD}\end{equation*}
commute.
\end{proof}

Suppose now that $X$ is logarithmicallly smooth. For every stratum $E$ of $X$ the fundamental group $\pi_1(E)$ operates on the (constant!) characteristic monoid $\Mbar_E$ on $E$ and therefore on the cone $\sigma_E=\Hom(\Mbar_E,\R_{\geq 0})$. Denote by $H_E$ the image of $\pi_1(E)$ in $\Aut(\Mbar_E)=\Aut(\sigma_E)$. Whenever $E$ is in the closure of a stratum $E'$ (but not equal to $E'$) we have a proper face morphism $\sigma_{E'}\rightarrow\sigma_{E}$ that is induced by the specialization map $\Mbar_{E}\rightarrow\Mbar_{E'}$. 

\begin{proposition}[\cite{AbramovichCaporasoPayne_tropicalmoduli} Proposition 6.2.6]
Let $X$ be a logarithmically smooth scheme locally of finite type over $k$. Then the generalized cone complex of $X$ is given as the colimit
\begin{equation*}
\Sigma_X=\lim_{\longrightarrow} \sigma_E/H_E 
\end{equation*} 
taken over all logarithmic strata of $X$. 
\end{proposition}

\begin{proof}
Let $U\rightarrow X$ be a strict \'etale morphism from a small logarithmic scheme $U$. Then there is a unique closed stratum $E$ of $X$ the closed stratum of $U$ is mapped into and we have $\sigma_E=\sigma_U=\Hom(P_U,\R_{\geq 0})$. There are two classes of face morphisms in the diagram defining $\Sigma_X$:
\begin{itemize}
\item The operation of $\pi_1(E)$ induces precisely the automorphisms of $\sigma_U$ already contained in $H_E$. 
\item Moreover, whenever $E$ is in the closure of a stratum $E'$ and $U'\rightarrow U$ is a strict \'etale morphism from a small logarithmic scheme such that the closed stratum of $U'$ is mapped into $E'$, the strict \'etale map $U'\rightarrow U$ induces the proper face morphism $\sigma_{E'}=\sigma_{U'}\rightarrow \sigma_{U}=\sigma_{E}$ coming from the specialization map $\Mbar_{E}\rightarrow\Mbar_{E'}$.
\end{itemize}
This gives a complete description of the diagram of face morphism that defines $\Sigma_X$ and so the claim follows.
\end{proof}

\subsection{Non-Archimedean skeletons of logarithmically smooth schemes}

In this section we prove Theorem \ref{thm_skel=trop}. Let $X$ be a logarithmically smooth scheme locally of finite type over $k$. Then, as seen in Corollary \ref{cor_logsmooth=toroidal}, the logarithmic scheme $X$ defines a toroidal embedding $X_0\hookrightarrow X$. The main result of \cite{Thuillier_toroidal} is that there is a strong deformation retraction $\mathbf{p}\mathrel{\mathop:}X^\beth\rightarrow X^\beth$ onto the \emph{skeleton} $\mathfrak{S}(X)$ of $X^\beth$ that depends on the toroidal structure of $X_0\hookrightarrow X$.

%Suppose first that $M_X$ is defined on $X_{Zar}$. By \cite[Lemme 3.14 and Corollaire 3.13]{Thuillier_toroidal} we can reduce our claim to the case of 
Suppose first that $X=\Spec k[P]$ an affine toric variety. Here the deformation retraction map $\mathbf{p}\mathrel{\mathop:}X^\beth\rightarrow X^\beth$ is given by sending $x\in X^\beth$ to the seminorm
\begin{equation*}\begin{split}
k[P]&\longrightarrow\R_{\geq 0}\\
\sum_{p\in P}a_p\chi^p&\longmapsto \max\vert a_p\vert\vert\chi^p\vert(x)
\end{split}\end{equation*}
on $k[P]$ and so $x$ and $\mathbf{p}(x)$ agree, when restricted to $P$. Define a section $J_X\mathrel{\mathop:}\Sigmabar_P\rightarrow X^\beth$ of $\trop_P$ by sending $u\in \Sigmabar_P=\Hom(P,\overline{\R}_{\geq 0})$ to the seminorm 
\begin{equation*}\begin{split}
J_X(u)\mathrel{\mathop:}k[P]&\longrightarrow\R_{\geq 0}\\
\sum_{p\in P}a_p\chi^p&\longmapsto \max\vert a_p\vert e^{-u(p)} 
\end{split}\end{equation*}
and note that $J_X$ is well-defined, continuous, and fulfills $\trop_P\circ J_X=\id$. Every element in $\mathfrak{S}(X)$ arises this way and, since $\Sigmabar_P$ is compact, the tropicalization map induces a homeomorphism $J_X\mathrel{\mathop:}\overline{\Sigma}_X\rightarrow \mathfrak{S}(X)$ making  the diagram
\begin{center}\begin{tikzpicture}
  \matrix (m) [matrix of math nodes,row sep=1em,column sep=2em,minimum width=2em]
  {  
  & & \mathfrak{S}(X) \\ 
  X^\beth&  &\\
  & & \Sigmabar_X\\
  };
  \path[-stealth]
    (m-2-1) edge node [above] {$\mathbf{p}$} (m-1-3)
    (m-2-1) edge node [below] {$\trop_X$} (m-3-3)
    (m-3-3) edge node [right] {$J_X$} (m-1-3);
\end{tikzpicture}\end{center}
commute. 

Now consider a Zariski logarithmic scheme that is logarithmically smooth over $k$. By \cite[Corollaire 3.13]{Thuillier_toroidal} $\mathbf{p}_X$ is uniquely determined by its restriction to small open subsets $U$ of $X$. Choose a strict \'etale morphism $\gamma\mathrel{\mathop:}U\rightarrow Z$ into an affine $T$-toric variety $Z=\Spec k[P]$ such that the image of the closed stratum of $U$ is in the closed $T$-orbit of $Z$. By \cite[Proposition 3.7]{Thuillier_toroidal} the skeleton of $U^\beth$ is defined as $\frakS(U)=(\gamma^\beth)^{-1}\big(\frakS(Z)\big)$ and $\gamma^\beth$ induces a homeomorphism $\frakS(U)\xrightarrow{\sim}\frakS(Z)$. Since the tropicalization map naturally factors as 
\begin{equation*}\begin{CD}
U^\beth @>\gamma^\beth>>Z^\beth @>\trop_Z>>\Hom(P,\Rbar_{\geq 0}) \ ,
\end{CD}\end{equation*}
by Theorem \ref{thm_tropfunc} this yields the claim. 

Finally assume that $X$ is an \'etale logarithmic scheme. Then for every strict surjective \'etale morphism $U\rightarrow X$ from a Zariski logarithmic schemes $U$, the logarithmic scheme $R=U\times_XU$ is also defined in the Zariski topology, i.e. without monodromy. By \cite[Proposition 3.29]{Thuillier_toroidal} all diagrams
\begin{equation*}\begin{CD}
U^\beth @>\mathbf{p}_{U}>>\mathfrak{S}(U)\\
@VVV @VVV\\
X^\beth @>\mathbf{p}_X>> \mathfrak{S}(X)
\end{CD}\end{equation*}
commute and the universal property of colimits implies that there is a continuous map $J_X\mathrel{\mathop:}\Sigmabar_X\rightarrow\mathfrak{S}(X)$ such that the diagram 
\begin{center}\begin{tikzpicture}
  \matrix (m) [matrix of math nodes,row sep=1em,column sep=2em,minimum width=2em]
  {  
  & & \mathfrak{S}(X) \\ 
  X^\beth&  &\\
  & & \Sigmabar_X\\
  };
  \path[-stealth]
    (m-2-1) edge node [above] {$\mathbf{p}$} (m-1-3)
    (m-2-1) edge node [below] {$\trop_X$} (m-3-3)
    (m-3-3) edge node [right] {$J_X$} (m-1-3);
\end{tikzpicture}\end{center}
commutes. This is a homeomorphism, since by \cite[Proposition 3.31]{Thuillier_toroidal} the skeleton $\mathfrak{S}(X)$ is the colimit of $\mathfrak{S}(R)\rightrightarrows\mathfrak{S}(U)$ for every strict \'etale morphism $U\rightarrow X$ from a logarithmic scheme $U$ without monodromy.

%%%%%%%%%%%%%%%%%%%%%%%%%%%%%%%%%%%%%%%%%%%%%%%%%%%%%%

\section{Comparison with Kajiwara-Payne tropicalization}\label{section_toricvar}

Let $X=X(\Delta)$ be a normal toric variety with big torus $T=\Spec k[M]$ determined by a rational polyhedral fan $\Delta$ in $N_\mathbb{R}=\Hom(M,\mathbb{R})$. We refer to \cite{Fulton_toricvarieties} for standard notation and background on toric varieties. Consider the Kajiwara-Payne extended tropicalization map 
\begin{equation*}
\trop_\Delta\mathrel{\mathop:}X^{an}\longrightarrow N_\R(\Delta)
\end{equation*} 
as defined in \cite[Section 1]{Kajiwara_troptoric} as well as \cite[Section 3]{Payne_anallimittrop} and alluded to in Section \ref{section_introtroptoric}. 

\begin{proposition}\label{prop_troplog=troptoric}
For a toric variety $X=X(\Delta)$ we have $\Sigma_X=\Delta$, the extended cone complex $\overline{\Sigma}_X$ is the closure of $\Delta$ in $N_\mathbb{R}(\Delta)$, and there is a commutative diagram
\begin{equation*}\begin{CD}
X^\beth @>\trop_X>> \overline{\Sigma}_X\\
@V\subseteq VV @VV\subseteq V\\
X^{an} @>\trop_\Delta>> N_\mathbb{R}(\Delta) \ .
\end{CD}\end{equation*}
\end{proposition}

\begin{proof}
It is enough to check the assertions on $T$-invariant open affine subsets $U_\sigma=\Spec k[S_\sigma]$. We can naturally identify
\begin{equation*}
\sigma = \big\{u\in\Hom(S_\sigma,\mathbb{R})\big\vert u(s)\geq 0 \ \ Ê \forall s\in S_\sigma\big\} = \Hom(S_\sigma,\mathbb{R}_{\geq 0})
\end{equation*}
and $\Hom(S_\sigma,\overline{\mathbb{R}}_{\geq 0})$ is the closure of $\sigma$ in $N_\mathbb{R}(\sigma)=\Hom(S_\sigma,\overline{\mathbb{R}})$. Under these identifications the identity $\trop_\Delta(x)=\trop_X(x)$ holds for all $x\in U_\sigma^\beth\subseteq U_\sigma^{an}$.
\end{proof}

Let $Y$ be a closed subset of $X$; for simplicity we assume $Y\cap T\neq\emptyset$ throughout. Its associated \emph{tropical variety} $\Trop(Y,\Delta)$ is defined to be the image under $\trop_\Delta$ of the closed subspace $Y^{an}$ in $N_\mathbb{R}(\Delta)$; one may alternatively characterize $\Trop(Y,\Delta)$ as the closure of $\trop_\Delta\big((Y\cap T)^{an}\big)$ in $N_\mathbb{R}(\Delta)$. 

\begin{corollary}
Given a closed subset $Y$ of $X$, we have the identity $\Trop_X(Y)=\Trop(Y,\Delta)\cap\overline{\Delta}$. 
\end{corollary}

\begin{proof}
This is an immediate consequence of Proposition \ref{prop_troplog=troptoric}.
\end{proof}

\begin{corollary} \label{cor_proper}
The closed subset $Y$ is proper over $k$ if and only if $\Trop_X(Y)=\Trop(Y,\Delta)$.
\end{corollary}

\begin{proof}
If $Y$ is proper over $k$, we have $Y^\beth=Y^{an}$ and therefore $\Trop(Y,\Delta)\subseteq \overline{\Delta}$. Then Proposition \ref{prop_troplog=troptoric} implies the claim. By \cite[Proposition 2.3]{Tevelev_tropcomp} (over $\mathbb{C}$) and \cite[Proposition 11.12]{Gubler_guide} (over all base fields $k$) $Y$ is proper over $k$, if and only if $\Trop(Y,\Delta)\subseteq \overline{\Delta}$. So the converse is true as well.  
\end{proof}

One may endow $Y$ with the logarithmic structure that is given as the pullback of $i^\ast M_X$ via the inclusion $i\mathrel{\mathop:}Y\rightarrow X$. By Theorem \ref{thm_tropfunc} there is a piecewise $\Z$-linear morphism $\Sigmabar(i)\mathrel{\mathop:}\Sigmabar_Y\rightarrow\Sigmabar_X$ such that $\Sigmabar(i)\big(\Trop_Y(Y)\big)=\Trop_X(Y)$. As can be seen in the following Example \ref{example_genericconic} this morphism is in general not injective. 

\begin{example}\label{example_genericconic}
Let $Y$ be a generic conic that intersects the toric boundary of $X=\PP^2$ as indicated below.
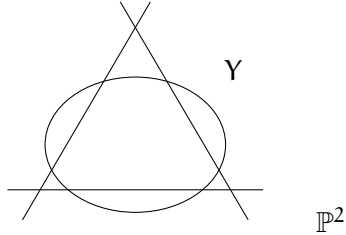
\begin{center}\begin{tikzpicture}
%\draw (-0.5,-0.3) -- (3.7,-0.3) -- (4.3,3.3) -- (0.3,3.3) -- (-0.5,-0.3);
\draw (0,0.4) -- (3.4,0.4);
\draw (0.2,0) -- (1.9,2.9);
\draw (3.2,0) -- (1.5,2.9);
\draw (1.7,1) ellipse (1.2cm and 0.9cm);
\node at (3,2) {$Y$};
\node at (4.3,0) {$\PP^2$};
\end{tikzpicture}\end{center}

Consider $Y$ as a logarithmic scheme with respect to the pullback logarithmic structure $i^\ast M_X$. Then the Kato fan $F_Y$ is given by  six copies of $\Spec\N$ glued over the generic points and the extended cone complex $\Sigmabar_Y$ consists of six copies of $\Rbar_{\geq 0}$ glued at the origin. The tropical variety $\Trop_X(Y)$ relative to $X$ is given by collapsing each of the pairs of cones to one (with tropical multiplicity $2$) and is equal to the $1$-skeleton of $\Sigmabar_{\PP^2}$.

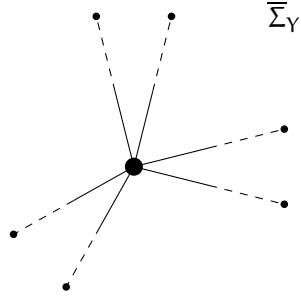
\begin{center}\begin{tikzpicture}
\draw (0,0) -- (1,0.25);
\draw [dashed] (1,0.25) -- (2,0.5);
\draw (0,0) -- (1,-0.25);
\draw [dashed] (1,-0.25) -- (2,-0.5);
\draw (0,0) -- (0.25,1);
\draw [dashed] (0.25,1) -- (0.5,2);
\draw (0,0) -- (-0.25,1);
\draw [dashed] (-0.25,1) -- (-0.5,2);
\draw (0,0) -- (-0.8,-0.45);
\draw [dashed] (-0.8,-0.45) -- (-1.6,-0.9);
\draw (0,0) -- (-0.45,-0.8);
\draw [dashed] (-0.45,-0.8) -- (-0.9,-1.6);

\fill (2,0.5) circle (0.05cm);
\fill (2,-0.5) circle (0.05cm);
\fill (0.5,2) circle (0.05cm);
\fill (-0.5,2) circle (0.05cm);
\fill (-1.6,-0.9) circle (0.05cm);
\fill (-0.9,-1.6) circle (0.05cm);
\fill (0,0) circle (0.12cm);

\node at (2,2) {$\Sigmabar_Y$};
\end{tikzpicture}\end{center}
\end{example}

Let us now generalize the observation made in the above example. Suppose that $Y$ is a sch\"on subvariety of $X$, i.e. assume that the multiplication map 
\begin{equation*}
\mu\mathrel{\mathop:}T\times Y\longrightarrow X
\end{equation*}
is smooth. By \cite[Proposition 2.7]{Ulirsch_tropcomplogreg} this is equivalent to $Y$, with the logarithmic structure induced from $X$, being logarithmically smooth over $k$. 

\begin{proposition}\label{prop_faithfultrop}
The induced map $\Sigmabar(i)\mathrel{\mathop:} \Sigmabar_Y\rightarrow \Sigmabar_X$ is an isomorphism of $\Sigmabar_Y$ if and only if the intersection of $Y$ with every $T$-orbit in $X$ non-empty and irreducible (i.e. has multiplicity one).
\end{proposition}

\begin{proof}
The map $\Sigmabar(i)\mathrel{\mathop:} \Sigmabar_Y\rightarrow \Sigmabar_X$ is an isomorphism if and only if the embedding induces a one-to-one correspondence between the logarithmic strata of $Y$ and the $T$-orbits of $X$. This is the case precisely when the intersection of $Y$ with every $T$-orbit in $X$ is non-empty and irreducible. 
\end{proof}

\begin{proof}[Proof of Corollary \ref{cor_schoen&faithful}]
Since $Y$ is proper, we have $Y^\beth=Y^{an}$ and, as explained in Corollary \ref{cor_proper} above, this also means that $\Trop_\Delta(Y)=\Sigmabar_X$. By Proposition \ref{prop_faithfultrop} we have $\Sigmabar_Y\simeq\Sigmabar_X$ if and only if the intersection of $Y$ with every $T$-orbit in $X$ is non-empty and irreducible, and therefore Theorem \ref{thm_skel=trop} yields the claim.
\end{proof}

%%%%%%%%%%%%%%%%%%%%%%%%%%%%%%%%%%%%%%%%%%%%%%%%%%%%%%

%%%%%%%%%%%%%%%%%%%%%%%%%%%%%%%%%%%%%%%%%%%%%%%%%%%%%%

%%%%%%%%%%%%%%%%%%%%%%%%%%%%%%%%%%%%%%%%%%%%%%%%%%%%%%

\bibliographystyle{amsalpha}
\bibliography{biblio}{}

\end{document}